\RequirePackage{fix-cm}
\documentclass[smallextended]{svjour3}       % onecolumn (second format)
\smartqed  % flush right qed marks, e.g. at end of proof
\usepackage{graphicx}
\usepackage{amsmath}
\usepackage{amssymb}
\usepackage{subfig}
\usepackage{algorithm}
\usepackage{algpseudocode}
\usepackage{paralist}

\newcommand{\dist}{\mathop{\rm dist}\nolimits}
\newcommand{\prox}{\mathop{\rm prox}\nolimits}
\newcommand{\amp}{\mathop{\:\:\,}\nolimits}
\newcommand{\Real}{\mathbb{R}}
\newcommand{\ba}{\boldsymbol{a}}
\newcommand{\bb}{\boldsymbol{b}}

\newcommand{\be}{\boldsymbol{e}}

\newcommand{\bp}{\boldsymbol{p}}

\newcommand{\bs}{\boldsymbol{s}}

\newcommand{\bu}{\boldsymbol{u}}

\newcommand{\bx}{\boldsymbol{x}}
\newcommand{\by}{\boldsymbol{y}}
\newcommand{\bz}{\boldsymbol{z}}

\newcommand{\bM}{\boldsymbol{M}}

\newcommand{\bX}{\boldsymbol{X}}

\newcommand{\balpha}{\boldsymbol{\alpha}}
\newcommand{\bbeta}{\boldsymbol{\beta}}

\newcommand{\bepsilon}{\boldsymbol{\epsilon}}

\newcommand{\btheta}{\boldsymbol{\theta}}

\newcommand{\bxi}{\boldsymbol{\xi}}

\newcommand{\bXi}{\boldsymbol{\Xi}}

\begin{document}

\title{Distance Majorization and Its Applications%\thanks{Grants or other notes
%about the article that should go on the front page should be
%placed here. General acknowledgments should be placed at the end of the article.}
}

\author{Eric C. Chi         \and
        Hua Zhou \and \\
        Kenneth Lange}

\authorrunning{E. Chi et al.} % if too long for running head

\institute{E. Chi \at
Department of Human Genetics, University of California, Los Angeles, CA 90095   \\
\email{ecchi@ucla.edu} \\
\and
              H. Zhou \at
Department of Statistics, North Carolina State University, Raleigh, NC 27695-8203 \\
\email{hua\_zhou@ncsu.edu} \\
\and
	     K. Lange  \at
Departments of Biomathematics, Human Genetics, and Statistics, University of California,
Los Angeles, CA 90095   \\
\email{klange@ucla.edu} \\
}

\date{Received: date / Accepted: date}
% The correct dates will be entered by the editor

\maketitle

%% ----------------------------------------------------------------------
%% Abstract
%% ----------------------------------------------------------------------
\begin{abstract}
The problem of minimizing a continuously differentiable convex function over an intersection of closed convex sets is ubiquitous in applied mathematics. It is particularly interesting when it is easy to project onto each separate set, but nontrivial to project onto their intersection. Algorithms based on Newton's method such as the interior point method are viable for small to medium-scale problems. However, modern applications in statistics, engineering, and machine learning are posing problems with potentially tens of thousands of parameters or more. We revisit this convex programming problem and propose an algorithm that scales well with dimensionality. Our proposal is an instance of a sequential unconstrained minimization technique and revolves around three ideas: the majorization-minimization (MM) principle, the classical penalty method for constrained optimization, and quasi-Newton acceleration of fixed-point algorithms. The performance of our distance majorization algorithms is illustrated in several applications.
% \PACS{PACS code1 \and PACS code2 \and more}
% 65K05 Mathematical programming methods
% 90C25 Convex programming
% 90C30 Nonlinear programming
% 62J02  General nonlinear regression
\end{abstract}

%% ----------------------------------------------------------------------
%% Introduction
%% ----------------------------------------------------------------------
\section{Introduction}
A wide spectrum of problems in applied mathematics and statistics can be formulated as an instance of the convex programming problem
\begin{eqnarray}
\label{eq:opt_prob}
\underset{\bx \in \cap_i C_i}{\min} \; \ell(\bx),
\end{eqnarray}
where $\ell(\bx)$ is continuously differentiable and convex and the $C_i$ are closed convex sets in $\Real^p$. At one extreme, problem (\ref{eq:opt_prob}) includes classical least squares. At the other, it includes finding a feasible point in the intersection of several closed convex sets. In between, the formulation covers a variety of shape restricted regression problems such as fitting a support vector machine and projecting an exterior point onto a complicated convex set. Great progress has been made in attacking specific incarnations of problem (\ref{eq:opt_prob}). The projected gradient algorithm and its Newton and quasi-Newton extensions have been very successful when constraints are simple, for example box constraints, and admit a correspondingly simple projection operator \cite{Ber1982,KimSraDhi2010,Ros1960,SchBerFri2009}. However, there is still room for improvement. In the current paper we present a unified approach to solving a smoothed relaxation of problem (\ref{eq:opt_prob}) via the majorization-minimization (MM) principle \cite{Lan2010}. This approach is especially attractive when it is easy to project onto each separate set $C_i$ but nontrivial to project onto their intersection.

Problem (\ref{eq:opt_prob}) can be written as the unconstrained optimization problem
\begin{eqnarray}
\label{eq:opt_prob_smooth}
\underset{\bx}{\min} \; \ell(\bx) + \sum_{i} \delta_{C_i}(\bx),
\end{eqnarray}
where the indicator function $\delta_C(\bx)$ equals 0 if $\bx \in C$ and $\infty$ if not. Although problem (\ref{eq:opt_prob_smooth}) is now unconstrained, the indicator functions introduce two challenges. The new objective function takes on infinite values and is non-differentiable. This prompts us to replace $\delta_C(\bx)$ by a finite valued smooth approximation $\dist(\bx,C)^2 = \inf_{\by \in C} \|\bx-\by\|_2^2$, where $\| \cdot \|_2$ denotes the
standard Euclidean norm. Further progress can be made by solving the related problem
\begin{eqnarray}
\label{eq:quadratic_penalty}
\underset{\bx}{\min} \; \ell(\bx) + \frac{\mu}{2} \sum_{i=1}^m \dist(\bx,C_i)^2,
\end{eqnarray}
where $\mu$ is a positive parameter that penalizes deviation from the original feasible region. The smooth approximation introduced in formulation (\ref{eq:quadratic_penalty}) is an example of
the quadratic penalty method \cite{Ber2003,NocWri2006,Rus2006}. Problem $(\ref{eq:quadratic_penalty})$ has many appealing features. The problem is unconstrained with an objective function that is convex and differentiable when $\ell(\bx)$ is convex and differentiable. Consequently optimality conditions can be readily identified. The distance function is closely tied to the
projection $P_C(\bx)$ of $\bx$ onto $C$; specifically $\dist(\bx,C) = \| \bx - P_C(\bx) \|_2$, and $\nabla \dist(\bx,C)^2 = 2[\bx - P_C(\bx)]$. Thus, a point $\bx$ solves problem $(\ref{eq:quadratic_penalty})$ if and only if it satisfies the stationarity condition
\begin{eqnarray*}
{\mathbf 0} = \nabla \ell(\bx) + \mu \sum_i [\bx - P_{C_i}(\bx)].
\end{eqnarray*}

Of course finding such an $\bx$ is often analytically intractable due to the projection term.
To solve (\ref{eq:quadratic_penalty}) iteratively, we resort to the MM principle. Because we rely on majorizing $\dist(\bx, C)$, we call our approach distance majorization.
A key step in solving the subproblems will be calculating projection operators. Fortunately, many useful projection operators are easy to compute.
The best known examples include projection
onto: (a) a closed Euclidean ball, (b) a closed rectangle, (c) 
a hyperplane, (d) a closed halfspace,
(e) a vector subspace, (f) the set of positive semidefinite matrices,
(g) the unit simplex, (h) a closed $\ell_1$ ball, and (i) an
isotone convex cone. While there are no analytic solutions
for the last three projections, there are efficient algorithms for computing
them \cite{BarBarBre1972,DucShaSin2008,Mic1986,RobWriDyk1988,SilSen2005}.

The rest of the paper is organized as follows. 
After a brief review of the MM principle and its place among related iterative minimization schemes, we illustrate the virtues of distance majorization
in five different problem areas: (a) finding a point in the intersection
of a finite collection of closed convex sets,
(b) projection of
a point onto the closest point in the intersection of a finite collection of closed
convex sets, (c) convex regression, (d) classification via support vector machines, and (e)
the facilities location problem. The literature 
on some of the examples is enormous, so we apologize in advance for omitting
relevant references and slighting the ramifications
of the various models.  After our tour of examples, we present relevant convergence
theory in a general algorithmic framework. Our concluding discussion indicates
a few extensions and limitations of distance majorization.

%% ----------------------------------------------------------------------
%% MM Principle + Distance Majorization
%% ----------------------------------------------------------------------
\section{The MM Principle and Distance Majorization}

Although first articulated by the numerical analysts Ortega and Rheinboldt \cite{OrtRhe1970}, the MM principle currently enjoys its greatest vogue in computational statistics \cite{BecYanLan1997,LanHunYan2000}. The basic idea is to convert a hard optimization problem (for example, non-differentiable) into a sequence of simpler ones (for example, smooth). The MM principle requires majorizing the objective function $f(\by)$ by a surrogate function $g(\by \mid \bx)$ anchored at the current point $\bx$.  Majorization is a combination of the tangency condition $g(\bx \mid \bx) =  f(\bx)$ and the domination condition $g(\by \mid \bx)  \geq f(\by)$ for all $\by \in \Real^p$.  The iterates of the associated MM algorithm are defined by
\begin{equation}
  \label{eq:generic-MM-iterate}
  \bx_{n+1} := \underset{\by}{\arg \min}\; g(\by \mid \bx_{n}).
\end{equation}
Because 
\begin{equation}
  f(\bx_{n+1}) \leq g(\bx_{n+1} \mid \bx_{n}) \leq g(\bx_{n} \mid \bx_{n}) = f(\bx_{n}),
\end{equation}
the MM iterates generate a descent algorithm driving the objective function downhill. Strict inequality usually prevails
unless $\bx_n$ is a stationary point of $f(\bx)$.

The most useful majorization of $\dist(\bx,C)$ follows immediately from the observations
\begin{eqnarray*}
\dist(\bx,C) & \le & \|\bx-P_C(\bx_n)\|_2 , \quad \mbox{and} \quad
\dist(\bx_n,C) = \|\bx_n-P_C(\bx_n)\|_2
\end{eqnarray*}
for all pairs $\bx$ and $\bx_n$. In practice, majorizing $\dist(\bx,C)^2$
by $\|\bx-P_C(\bx_n)\|_2^2$ leads to more convenient updates than
majorizing $\dist(\bx,C)$ by $\|\bx-P_C(\bx_n)\|_2$.
Most of our applications can be phrased as minimizing the criterion
\begin{eqnarray}
\label{eq:general_surrogate}
g(\bx \mid \bx_n) = \ell(\bx) + \frac{\mu}{2} \sum_{i=1}^m \gamma_i \|\bx-P_{C_i}(\bx_n)\|_2^2,
\end{eqnarray}
for a convex loss $\ell(\bx)$, a collection $\{C_1,\ldots,C_m\}$ of closed convex sets, a positive penalization parameter $\mu$,
and a corresponding set of positive weights $\gamma_1,\ldots,\gamma_m$. Without loss in generality, we can require $\gamma_i$ to sum to one, since scaling of the weights can be absorbed into the overall penalty parameter $\mu$. Uniform weights equally penalize an iterate's violation of each constraint. Nonuniform weights will penalize constraint violations differently. This can be a useful mechanism if it is more important to satisfy some constraints over others in an application. In this paper we consider examples where constraints are all equally important and consequently employ uniform weights. For notational simplicity, we drop the weights from the remainder of our exposition but note that they can be employed in all the examples we cover.  The introduction of weights also leaves the convergence analysis presented later untouched. Algorithm~\ref{alg:DistanceMajorization} shows the pseudocode for the distance majorization algorithm.

We highlight the fact that the algorithm does not require the projection onto the intersection but rather only the projection onto each of the constituent sets $C_i$. As we will see in our first example, distance majorization
can be considered a generalization of the simultaneous projection algorithm for finding a point in the intersection of a collection of closed convex sets. 
We note, however, that distance majorization is not unique in this regard. For comparison's sake, we will also present a
dual ascent algorithm at the end of the next section that employs projections onto the constituent sets. Although the two methods exhibit comparable empirical performance,
the distance majorization algorithm is guaranteed to converge under weaker conditions than the dual ascent algorithm.

\begin{algorithm}[t]
  \caption{Distance Majorization}
  \label{alg:DistanceMajorization}  
  \begin{algorithmic}[1]
  	\State Given $\mu_0 > 0$ and a starting point $\bx_0$.
	\State $k \leftarrow 0$
	\Repeat
	\State $\by \leftarrow \bx_k$
    \Repeat
    	\For{$i = 1,\ldots, m$}
				\State $\bp_i \leftarrow P_{C_i}(\by)$
		\EndFor
%		\State $\by \leftarrow \underset{\bu}{\arg\min} \; \ell(\bu) + \mu_k \sum_{i=1}^m \gamma_i \| \bu - \bp_i \|_2^2$
		\State $\by \leftarrow \underset{\bu}{\arg\min} \; \ell(\bu) + \mu_k \sum_{i=1}^m \| \bu - \bp_i \|_2^2$		
	 \Until{convergence}
	 \State Choose new penalty parameter $\mu_{k+1} > \mu_k$	 
	 \State $k \leftarrow k+1$
	 \State $\bx_k \leftarrow \by$
	 \Until{convergence}
  \end{algorithmic}
\end{algorithm}

%% ----------------------------------------------------------------------
%% quasi-Newton acceleration
%% ----------------------------------------------------------------------

Finally, we note that MM algorithms are often plagued by a slow rate of convergence in a neighborhood of the minimum point. To remedy this situation, we employ
quasi-Newton acceleration.  MM algorithms can be accelerated via Newton's method just as the classic gradient descent algorithm. Adjusting the direction of steepest descent to account for the curvature in the objective yields more efficient step directions, and the number of iterations to a minimum can be drastically reduced.
Newton's method, however, requires computing and storing a full Hessian matrix, a demanding task in high-dimensional problems. To ease the computational burden,
quasi-Newton methods obtain curvature information by approximating the Hessian with secants or differences between successive iterates. Using more secants leads to better approximations of the Hessian initially, but using too many secants can actually lead to a poorer approximation as a smaller collection of secants can adapt more dynamically to changes in the curvature as the iterations proceed. Moreover, using more secants entails additional storage and computation. In the following examples, we use either two or five secants. Using a handful of secants is a modest additional burden in computation and storage but leads to noticeable acceleration in our MM algorithm.
For details on the scheme we employed as well as comparisons with alternative acceleration schemes, we direct readers to our earlier paper \cite{ZhoAleLan2011}.

%% ----------------------------------------------------------------------
%% Relationship to Sequential Unconstrained Minimization
%% ----------------------------------------------------------------------
\subsection{Sequential Unconstrained Minimization}

During the review of this paper, a referee brought to our attention that the MM algorithm is an instance of a broad class of methods termed sequential unconstrained minimization \cite{FiaMcC1990}.  Consider minimizing $f(\bx) : \Real^p \rightarrow \Real$ over a closed, non-empty set $C \subset \Real^p$. In sequential unconstrained minimization, we generate a sequence of iterates that minimize an unconstrained surrogate
\begin{eqnarray*}
\bx_{n} = \underset{\bx}{\arg\min} \; G_n(\bx) := f(\bx) + h_n(\bx),
\end{eqnarray*}
where the auxiliary functions $h_n(\bx)$ encode information about the constraint set $C$. 

When $h_n(\bx)$ is chosen so that $h_n(\bx) \geq 0$ for all $\bx$ and $h_n(\bx_{n-1}) = 0$, then
\begin{eqnarray*}
f(\bx_n) \leq f(\bx_n) + h_n(\bx_n) = G_n(\bx_n) \leq G_n(\bx_{n-1})  = f(\bx_{n-1}).
\end{eqnarray*}
This is a restatement of the descent property of an MM algorithm. In fact, we can identify $G_n(\bx) = g(\bx \mid \bx_{n-1})$ and $h_n(\bx) = g(\bx \mid \bx_{n-1}) - f(\bx)$. The tangency and domination conditions of the MM principle can be expressed alternatively as
\begin{eqnarray*}
G_n(\bx) = g(\bx \mid \bx_{n-1}) = f(\bx) + [g(\bx \mid \bx_{n-1}) - f(\bx)] = f(\bx) + h_n(\bx),
\end{eqnarray*}
where $h_n(\bx) \geq 0$ and $h_n(\bx_{n-1}) = 0$.

Byrne \cite{Byr2008a} introduced an important subset of sequential unconstrained minimization methods in which the auxiliary functions $h_n(\bx)$ satisfy
\begin{eqnarray}
\label{eq:summa}
G_n(\bx) - G_n(\bx_n) \geq h_{n+1}(\bx) \geq 0.
\end{eqnarray}
Methods satisfying (\ref{eq:summa}) are examples of sequential unconstrained minimization algorithms (SUMMA) and generate iterates for which $f(\bx_n)$ converges to $\inf_{\bx \in C} f(\bx)$. 
The SUMMA class includes a wide range of general iterative methods including barrier and penalty function methods, forward-backward splitting methods, and instances of the expectation maximization (EM) algorithm to name a few. Readers can consult the references \cite{Byr2008a,Byr2013,Byr2013a} to learn more about the breadth and applicability of the SUMMA class.

Given that examples of the EM algorithms have been shown to belong to the SUMMA class \cite{Byr2013} and that EM algorithms are a special case of the MM algorithm \cite{WuLan2010}, it is natural to wonder if MM algorithms, which have now been shown to be sequential unconstrained minimization algorithms, belong to the SUMMA class. The answer to this question is not immediately obvious. It is possible to concoct majorizations that fail to meet the SUMMA condition. Rewriting the SUMMA condition (\ref{eq:summa}) in terms of majorizations yields
\begin{eqnarray}
\label{eq:summa_mm}
g(\bx \mid \bx_{n-1}) - g(\bx_n \mid \bx_{n-1}) \geq g(\bx \mid \bx_n) - f(\bx) \geq 0,
\end{eqnarray}
for all $\bx$. Roughly speaking, (\ref{eq:summa_mm}) says that a sequence of majorizations should be hugging $f(\bx)$ uniformly more closely as the iterations proceed. While this is intuitively desirable, it is not necessary to ensure convergence of an MM algorithm. 

Nonetheless, it can be non-trivial to declare an iterative algorithm to be outside the SUMMA class, since we must prove that the resulting iterative algorithm could not be derived from some sequence of auxiliary functions that do obey (\ref{eq:summa}). 
Although majorizations chosen may violate the SUMMA condition, the resulting iterative algorithm may ultimately belong to the SUMMA class. In the Appendix we give an example of a convergent MM algorithm with a surrogate function that fails condition (\ref{eq:summa_mm}) globally. Locally the algorithm does belong to the SUMMA class. The ambiguity about the
proper domain of an algorithm spills over into selection of starting points and highlights the practical benefits of the MM principle, which leaves the door ajar to less restrictive auxiliary functions. Fortunately, the qualitative features of convergence carry over to this broader set of auxiliary functions.

%Nonetheless, declaring that an iterative algorithm does not belong to the SUMMA
%class is non-trivial, because doing so entails proving that the resulting iterative
%algorithm could not be derived from some sequence of auxiliary functions that
%obey (\ref{eq:summa}). Indeed, the example we give in the Appendix, under some modest
%restrictions, can be seen to belong to the SUMMA class, even though the
%originally employed set of auxiliary functions fails the SUMMA condition.
%This ambiguity, however, highlights the practical benefits of the MM algorithm. The MM approach opens the door to a wider set of auxiliary functions
%from which to choose. One can construct algorithms using auxiliary functions
%that do not obey (\ref{eq:summa_mm}) and rely on the theory for MM algorithms to ensure
%convergence. At the end of the day, the resulting iterative procedure may indeed belong to the SUMMA class, but we do not necessarily have to fret over
%satisfying (\ref{eq:summa}) when engineering a solution.

%% ----------------------------------------------------------------------
%% Examples
%% ----------------------------------------------------------------------
\section{Examples of Distance Majorization}

%% ----------------------------------------------------------------------
%% 1. Feasible Point
%% ----------------------------------------------------------------------
\subsection*{Finding a Feasible Point}

When the intersection $C=\cap_{j=1}^m C_j$ is nonempty, majorization can be employed 
to locate a point in $C$.   The general idea is to drive the convex combination
\begin{eqnarray}
\label{eq:feasible_point}
%f(\bx) & = & \sum_{i=1}^m \gamma_i \dist(\bx,C_i)^2
f(\bx) & = & \sum_{i=1}^m \dist(\bx,C_i)^2
\end{eqnarray}
to 0.  Minimization of the surrogate function 
\begin{eqnarray*}
%g(\bx \mid \bx_n) & = & \sum_{i=1}^m \gamma_i \|\bx-P_{C_i}(\bx_n)\|_2^2
g(\bx \mid \bx_n) & = & \sum_{i=1}^m \|\bx-P_{C_i}(\bx_n)\|_2^2
\end{eqnarray*}
leads to the well-known simultaneous projection algorithm 
\begin{eqnarray*}
%\bx_{n+1} & = & \sum_{i=1}^m \gamma_i P_{C_i}(\bx_n) .
\bx_{n+1} & = & \sum_{i=1}^m P_{C_i}(\bx_n) .
\end{eqnarray*}
The earliest version of this algorithm is attributed to Cimmino \cite{Cim1938}.
It does not necessarily find the closest point in $C$ to $\bx$. The evidence 
suggests that simultaneous projection converges more slowly than alternating 
projection \cite{CenCheCom2012,Gou2008}. However, simultaneous projection enjoys 
the advantage of being parallelizable. One can invoke the theory of paracontractive 
operators to prove the convergence of both simultaneous and alternating projections 
\cite{Byr2008}.

The alternating projection algorithm can also be derived by distance majorization. 
The least distance between two closed convex sets $C_1$ and $C_2$ can be found
by minimimizing $\dist(\bx,C_2)^2$ over $\bx \in C_1$.  If we majorize $\dist(\bx,C_2)$ 
by the surrogate function $\|\bx-\by_n\|^2$, where $\by_n = P_{C_2}(\bx_n)$, then the 
minimum of the surrogate occurs at $P_{C_1}(\by_n)= P_{C_1}[P_{C_2}(\bx_n)]$.  When the two 
sets intersect, the least distance of 0 is achieved at any point in the intersection.
Thus, the MM principle provides very simple and direct derivations of the simultaneous and alternating projection algorithms.

Distance majorization can be generalized by replacing Euclidean distances with Bregman divergences. 
For simplicity we limit our discussion to Bregman divergences generated by 
strictly convex twice differentiable functions $\phi(\bx)$. The Bregman divergence
\begin{eqnarray*}
D_\phi(\by \mid \bx) & = & \phi(\by) - \phi(\bx) - \langle \nabla \phi(\bx), \by-\bx \rangle.
\end{eqnarray*}
is a convex function of $\by$ anchored at $\bx$ and majorizing 0. For instance, the four convex
functions $\phi_1(\by) = \|\by\|^2$, $\phi_2(\by)= -\sum_i \log y_i$, 
$\phi_3(\by) = \sum_i y_i \ln y_i$, and $\phi_4(\by \mid \bx)=\by^t \bM \by$ generate the Bregman divergences 
\begin{eqnarray*}
D_{\phi_1}(\by \mid \bx) & = & \|\by-\bx\|^2 \\
D_{\phi_2}(\by \mid \bx) & = & \sum_i \Big[{y_i \over x_i}-\log \Big({y_i \over x_i}\Big) -1\Big] \\
D_{\phi_3}(\by \mid \bx) & = & \sum_i y_i \ln \Big({y_i \over x_i}\Big) - \sum_i (y_i-x_i) \\
D_{\phi_4}(\by \mid \bx) & = & (\by-\bx)^t\bM (\by-\bx). 
\end{eqnarray*}
The matrix $\bM$ in the definition of $\phi_4(\by)$ is assumed positive definite.

The Bregman projection $P^\phi_C(\bx)$ onto a closed convex set $C$ is defined as
\begin{eqnarray*}
P^\phi_C(\bx) & = & \underset{\by \in C}{\arg\min} \; D_\phi(\by \mid \bx).
\end{eqnarray*}
Under suitable additional hypotheses, the Bregman projection exists. It
is unique because $\phi(\bx)$ is strictly convex. Moreover,
$P^\phi_C(\bx) = \bx$ (equivalently $D_\phi [P^\phi_C(\bx),\bx] = 0$)
exactly when $\bx \in C$. The analogue of the
proximity function (\ref{eq:feasible_point}) is the proximity function
\begin{eqnarray}
\label{eq:feasible_point_Bregman}
f(\bx) & = & \sum_{i=1}^m D_{\phi_i}[P^{\phi_i}_{C_i}(\bx) \mid \bx].
\end{eqnarray}
If we abbreviate $\by^i_n = P^{\phi_i}_{C_i}(\bx_n)$,
then the function
\begin{eqnarray*}
\varphi(\bx \mid \bx_n) & = & \sum_i D_{\phi_i}(\by^i_n \mid \bx),
\end{eqnarray*}
majorizes $f(\bx)$. The MM principle suggests that we minimize $\varphi(\bx \mid \bx_n)$.
A brief calculation produces the stationarity condition
\begin{eqnarray*}
{\bf 0} & = & \nabla \varphi(\bx \mid \bx_n) \amp = \amp \sum_i \nabla^2 \phi_i(\bx)(\bx - \by^i_n),
\end{eqnarray*}
where $\nabla^2 \phi_i(\bx)$ denotes the Hessian of $\phi_i(\bx)$. Readers can consult \cite{ByrCen2001} 
for a more in depth and thorough treatment of minimizing the proximity function (\ref{eq:feasible_point_Bregman}).

%% ----------------------------------------------------------------------
%% 2. Projection onto Intersection of Closed Convex Sets
%% ----------------------------------------------------------------------
\subsection*{Projection onto the Intersection of Closed Convex Sets}

We next consider how distance majorization can be used to find the closest point in the intersection $C$ to a point $\by$. 
This involves minimizing the strictly convex function
\begin{eqnarray*}
%f_{\mu}(\bx) & = &  {1 \over 2} \|\bx-\by\|^2+{\mu \over 2}\sum_{i=1}^m \gamma_i d_{C_i}(\bx)^2
f_{\mu}(\bx) & = &  {1 \over 2} \|\bx-\by\|^2+{\mu \over 2}\sum_{i=1}^m d_{C_i}(\bx)^2
\end{eqnarray*}
for $\mu$ large.  The solution $\bx(\mu)$ tends to the optimal point as $\mu$ tends
to $\infty$.  The MM update for minimizing the surrogate function
\begin{eqnarray*}
%g_{\mu}(\bx \mid \bx_n) & = &  {1 \over 2} \|\bx-\by\|^2+{\mu \over 2}\sum_{i=1}^m \gamma_i \|\bx-P_{C_i}(\bx_n)\|^2
g_{\mu}(\bx \mid \bx_n) & = &  {1 \over 2} \|\bx-\by\|^2+{\mu \over 2}\sum_{i=1}^m \|\bx-P_{C_i}(\bx_n)\|^2
\end{eqnarray*}
is the convex combination
\begin{eqnarray}
%\bx_{n+1} & = & {1 \over 1+\mu } \by+ {\mu \over 1+\mu }\sum_{i=1}^m \gamma_i P_{C_i}(\bx_n). \label{eqn:proj-MM}
\bx_{n+1} & = & {1 \over 1+\mu } \by+ {\mu \over 1+\mu }\sum_{i=1}^m P_{C_i}(\bx_n). \label{eqn:proj-MM}
\end{eqnarray}
The corresponding algorithm map $\psi(\bx) = \arg\min_{\bu} g_{\mu}(\bu \mid \bx)$ is strictly contractive with contraction constant $c=\mu/(1+\mu)$.  According to the contraction mapping theorem, the
iterates converge to the unique fixed point at geometric rate $c$. This fixed point coincides with the
minimum point of the function $f_{\mu}(\bx)$.  Indeed, $f_{\mu}(\bx)$ is differentiable with gradient
\begin{eqnarray*}
%\nabla f_{\mu}(\bx) & = & \bx-\by+ \mu \sum_{i=1}^m \gamma_i [\bx-P_{C_i}(\bx)] .
\nabla f_{\mu}(\bx) & = & \bx-\by+ \mu \sum_{i=1}^m [\bx-P_{C_i}(\bx)] .
\end{eqnarray*}
Rearrangement of the stationarity condition $\nabla f_{\mu}(\bx) ={\bf 0}$ gives
the fixed point condition
\begin{eqnarray*}
%\bx & = & {1 \over 1+\mu } \by+ {\mu \over 1+\mu }\sum_{i=1}^m \gamma_i P_{C_i}(\bx) .
\bx & = & {1 \over 1+\mu } \by+ {\mu \over 1+\mu }\sum_{i=1}^m P_{C_i}(\bx) .
\end{eqnarray*}

One can generalize these results in various ways. For instance, if we replace Euclidean loss by weighted Euclidean loss ${1 \over 2} \sum_{i=1}^p w_i (x_i-y_i)^2$, then the MM update of the penalized loss has components
\begin{eqnarray*}
%x_{n+1,i} & = & {w_i \over w_i+\mu } y_i+ {\mu \over w_i+\mu }\sum_{k=1}^m \gamma_k P_{C_k}(\bx_n)_i.
x_{n+1,i} & = & {w_i \over w_i+\mu } y_i+ {\mu \over w_i+\mu }\sum_{k=1}^m P_{C_k}(\bx_n)_i.
\end{eqnarray*}
The quadratic penalty method suffers from roundoff errors and numerical instability for large $\mu$. These are mitigated in the MM algorithm since its updates (\ref{eqn:proj-MM}) rely on stable projections and 
avoid matrix inversion. The slow rate $\mu/(1+\mu)$ of convergence for large $\mu$ is an issue. In practice
one can improve the rate of convergence by starting $\mu$ small and gradually increasing it to its target value. For a fixed $\mu$ one can also accelerate the MM iterates by systematic extrapolation. For instance, our quasi-Newton acceleration \cite{ZhoAleLan2011} often reduces the required number of iterations by one or two orders of magnitude. 

%% ----------------------------------------------------------------------
%% Projection via Dual Ascent
%% ----------------------------------------------------------------------
\subsubsection*{Projection as a Dual Program}

For the sake of comparison, we describe a dual algorithm for solving the projection problem. This alternative algorithm can be accelerated by Nesterov's method \cite{BecTeb2009,Nes2007}. The unaccelerated dual algorithm is a variation of Dykstra's algorithm \cite{Dyk1983}, which solves the dual problem by block descent. 

To derive the dual problem, we first observe that the primal problem consists of minimizing 
\begin{eqnarray*}
\frac{1}{2}\left\| \bx - \by \right\|_2^2 + \sum_{i=1}^m \delta_{C_i}(\bx_i)
\end{eqnarray*}
subject to $\bx_1 = \bx, \ldots, \bx_m = \bx$. The Lagrangian for the primal problem is
\begin{eqnarray*}
{\cal L}(\bx, \bx_1, \ldots, \bx_m, \bz_1, \ldots, \bz_m) 
& = & \frac{1}{2}\left\| \bx - \by \right\|_2^2 -\Big(\sum_{i=1}^m \bz_i\Big)^t\bx \\
& + & \sum_{i=1}^m [\delta_{C_i}(\bx_i) + \bz_i^t\bx_i].
\end{eqnarray*}
If $\bz =  (\bz_1, \ldots, \bz_m)$ denotes the concatenation of the dual variables $\bz_i$,
then the dual function  
\begin{eqnarray*}
{\cal D}(\bz) & = & \underset{\bx, \bx_1, \ldots, \bx_m}{\inf} {\cal L}(\bx, \bx_1, \ldots, \bx_m, \bz_1, \ldots, \bz_m) 
\end{eqnarray*}
reduces to
\begin{eqnarray*}
{\cal D}(\bz) & = - \frac{1}{2}\|\bs\|_2^2 - \bs^t\by - \sum_{i=1}^m \sup_{\bx_i \in C_i } 
(-\bz_i^t \bx_i ),
\end{eqnarray*}
where $\bs = \sum_{i=1}^m \bz_i$. The dual function can be maximized by the proximal gradient algorithm
\begin{equation*}
\begin{split}
\bx^{n} &\leftarrow \by + \sum_{i=1}^m \bz^{n}_i \\
\bz^{n+1}_i &\leftarrow \bz^{n}_i +  [P_{C_i}(\bx^{n} - \bz^{n}_i) - \bx^{n}]. \\
\end{split}
\end{equation*}
For the sake of clarity, we have adopted novel notation in this derivation; $\bx^{n}_i$ and $\bz^{n}_i$ denote the $n$th MM iterate of the $i$th primal and dual variables respectively.

Derivation of this algorithm and its Nesterov acceleration (FISTA) appear in the Appendix. The dual
updates, which are essentially projection steps, can be computed in parallel.  Thus, the dual algorithm matches the MM algorithm in this regard.

%% ----------------------------------------------------------------------
%% 2a. Projection onto Doubly Nonnegative Matrices
%% ----------------------------------------------------------------------
\subsubsection*{Projecting onto the set of doubly nonnegative matrices}

As a numerical example, consider the problem of projecting a symmetric matrix onto the set of doubly nonnegative matrices, namely the intersection of the set of nonnegative matrices with the set of positive semi-definite matrices. Many covariance matrices -- for example, kinship matrices in statistical genetics -- have nonnegative entries. Projection onto each of the component sets is relatively easy while projection onto the intersection is not. Projecting onto the set of nonnegative matrices is accomplished by setting all negative entries of a matrix to zero. Projecting onto the set of positive semi-definite matrices is accomplished by truncating the eigenvalue decomposition of the matrix and rejecting all outer products with negative eigenvalues.

As a test case, we generated a 200-by-200 matrix with independent and identically distributed (i.i.d.) entries drawn from a standard normal distribution. After projecting the simulated matrix onto the space of symmetric matrices, we compared 
the distance majorization algorithm to its quasi-Newton acceleration (2 secants), the dual proximal gradient algorithm, and its FISTA acceleration.  We implemented the MM algorithm with the geometrically increasing sequence $\mu_i = 2^i - 1$ of penalty constants $\mu$. The decision to switch to the next larger $\mu$ was based on the ratio
\begin{eqnarray}
\label{eq:stopping_rule}
\frac{\lVert \bx^{n+1} - \bx^{n} \rVert_2}{\lVert \bx^{n} \rVert_2 + 1}
\end{eqnarray}
Whenever this ratio fell below $\rho = 10^{-4}$, we updated $\mu$. To track the progress of each algorithm, we calculated two measures of constraint violation by the current matrix: (a) the absolute value of the most negative eigenvalue, and (b) the absolute value of the most negative entry. Figure~\ref{fig:dnn} plots the maximum of the two constraint violations on a log scale at each iteration. The abrupt transitions in the MM and quasi-Newton MM paths reflect the switch points for the penalty constant $\mu$. Obviously, the amount of work done in each iterate varies across the methods. For a more direct comparison, Table~\ref{tab:dnn} records several statistics, including run times in seconds.  In the table, the distance column conveys the Frobenius norm of the difference between the simulated matrix and the fitted matrix. The two featured algorithms perform about equally well. As expected, their accelerated versions do much better.

\begin{figure}
\centering
\includegraphics[scale=0.45]{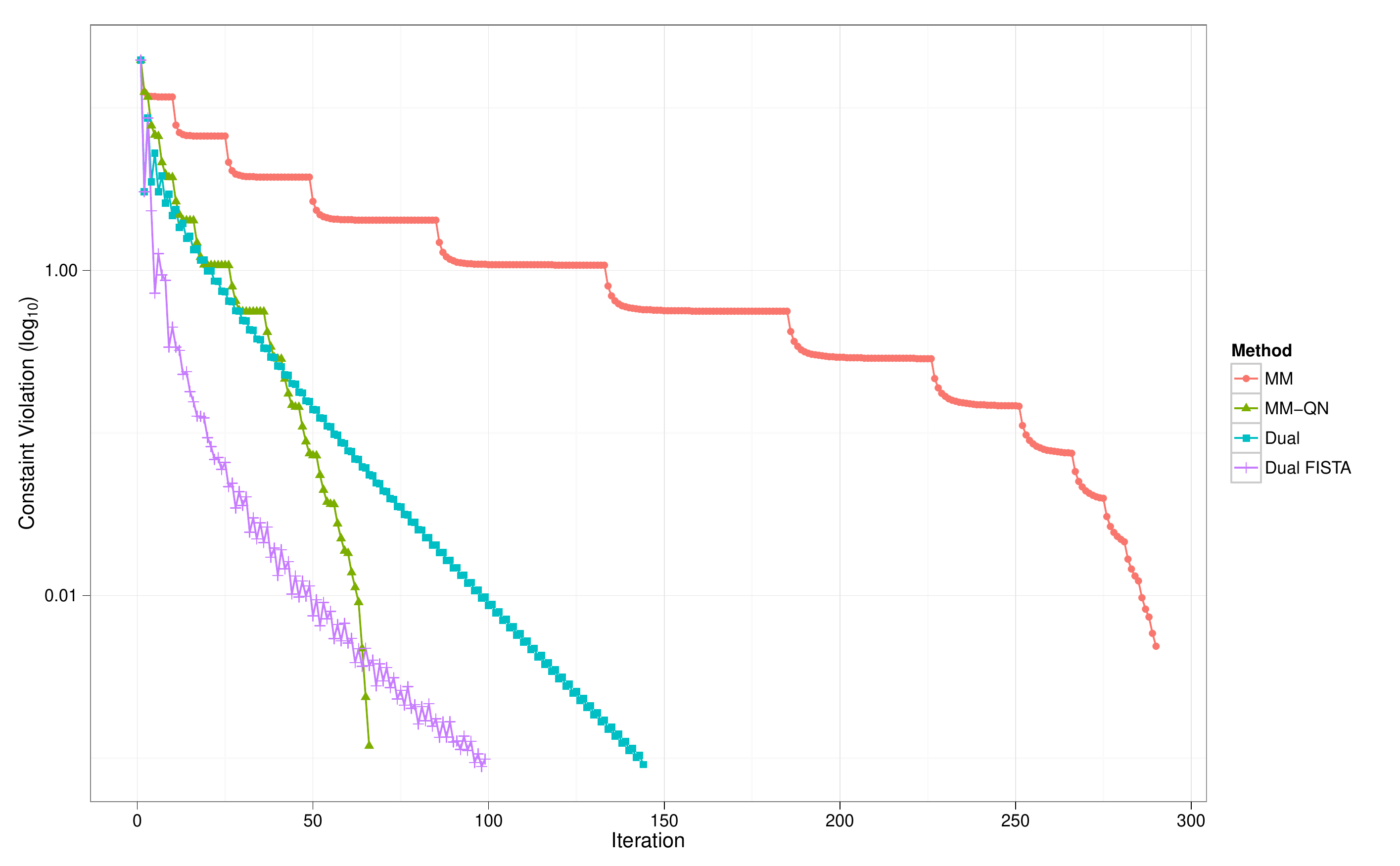}
\caption{A comparison of the MM algorithm, its quasi-Newton acceleration, the dual proximal gradient algorithm, and its FISTA acceleration applied to the problem of projecting a $200\times200$ matrix onto the set of doubly nonnegative matrices.
}\label{fig:dnn}
\end{figure}

% latex table generated in R 2.15.1 by xtable 1.7-0 package
% Mon Aug  6 12:24:47 2012
\begin{table}[ht]
\caption{Timing comparisons and constraint violations for projecting onto the set of doubly nonnegative matrices.}
\label{tab:dnn}
\begin{center}
\begin{tabular}{lrrrr}
\hline\noalign{\smallskip}
 Method & Time (sec) & Iterations & Distance & Constraint Violation \\ 
\noalign{\smallskip}\hline\noalign{\smallskip}
 MM & 16.526 & 290 & 120.9110 & -0.0048710012 \\ 
 MM-QN & 11.098 & 98 & 120.9131 & -0.0007433297 \\ 
 Dual & 19.882 & 144 & 120.9131 & -0.0009122053 \\ 
 Dual (Acc.) & 13.926 & 99 & 120.9136 & -0.0009862162 \\ 
   \noalign{\smallskip}\hline
\end{tabular}
\end{center}
\end{table}

%% ----------------------------------------------------------------------
%% 2b. Isotonic Regression
%% ----------------------------------------------------------------------
\subsubsection*{Shape-Restricted Regression}

Isotone regression minimizes the least squares criterion $\frac{1}{2}\sum_{i=1}^n w_i (y_i - x_i)^2$ subject to the isotonic constraint $x_1 \le \cdots \le x_n$.
This problem is readily amenable to the projection algorithm. Projection onto the isotone convex cone
\begin{eqnarray*}
C & = & \{\bx  : x_1 \le x_2 \le \cdots \le x_n\}
\end{eqnarray*}
is rapidly accomplished by the pool adjacent violators algorithm \cite{BarBarBre1972,RobWriDyk1988,SilSen2005}.   More complicated order restrictions such as $x_i \le x_j$ for all arcs $(i,j)$ in a directed graph can be handled as well.  In this setting all components of a vector $\bx$ projected on the convex set $C_{ij} = \{\bx: x_i \le x_j\}$ are left untouched except components $x_i$ and $x_j$, These are left untouched when $x_i \le x_j$. Both $x_i$ and $x_j$ are replaced by their average when $x_i > x_j$.

\begin{figure}
\centering
\includegraphics[scale=0.45]{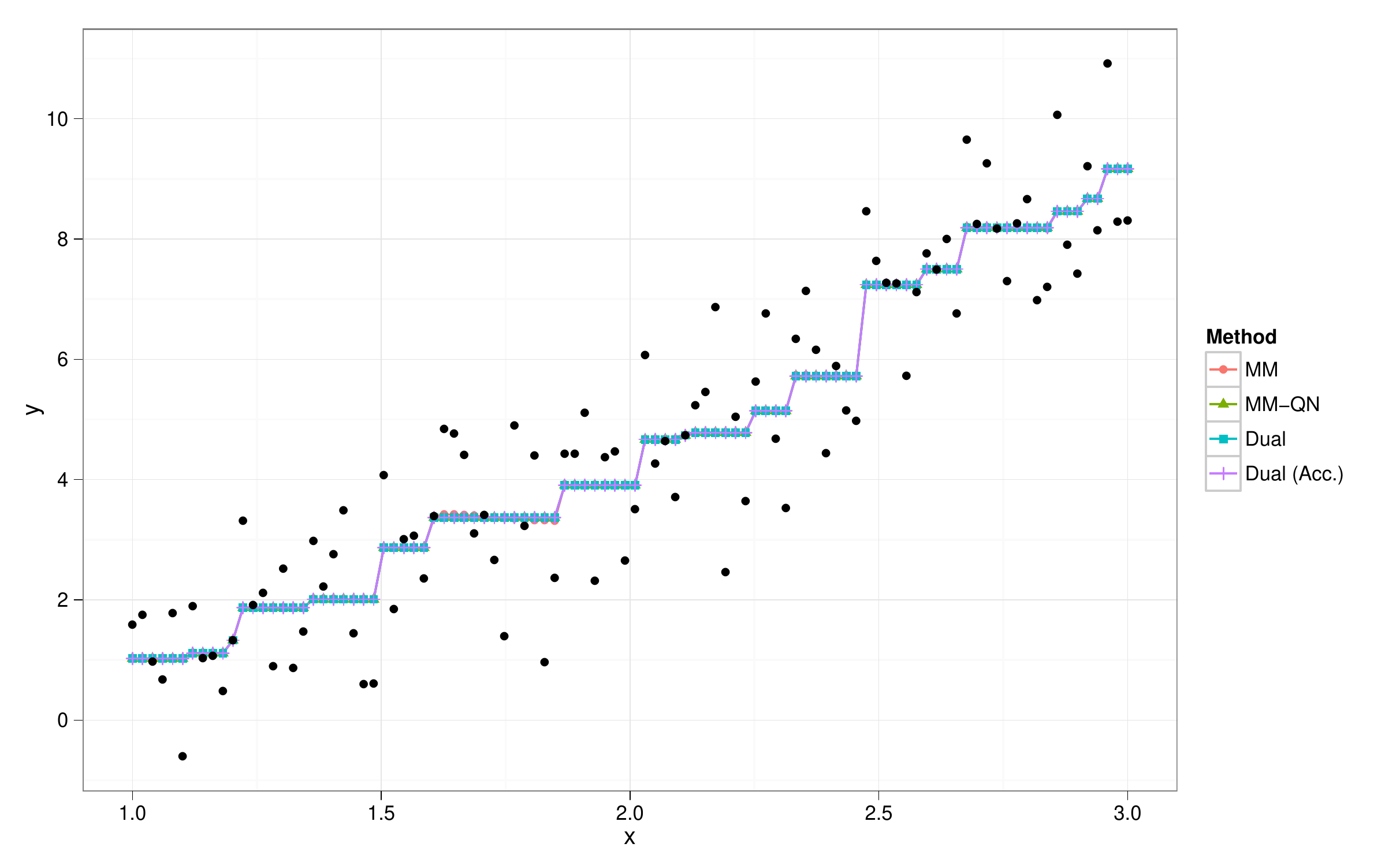}
\caption{Fitted data for isotonic regression.
\label{fig:isotone_data}}
\end{figure}

We considered the problem of fitting a nondecreasing function to the data shown in Figure~\ref{fig:isotone_data} (black dots). Each observed pair $(x_i,y_i)$ was generated as follows. The $x_i$ are equally spaced points between 1 and 3, and the $y_i$ satisfy
\begin{eqnarray*}
y_i & =  & x_i^2 + \epsilon_i,
\end{eqnarray*}
where the $\epsilon_i$ are i.i.d.\ standard normal deviates. For the MM algorithms we used the geometrically increasing sequence of penalty constants $\mu_i$ featured in the previous example and two secant conditions for the quasi-Newton acceleration.  We switched to the next value of $\mu$ whenever the stopping criterion (\ref{eq:stopping_rule}) fell below $\rho = 10^{-6}$. A looser threshold $\rho = 10^{-4}$ resulted in unacceptably poor fits for these data.

To track the progress of each algorithm, we measured the constraint violation of an iterate as the maximum absolute constraint violation between two successive parameters. Figure~\ref{fig:isotone_data} shows that all four algorithms return similar solutions under the specified stopping rule. Figure~\ref{fig:isotone} plots the constraint violation for each method on a log scale. Table~\ref{tab:isotone} compares timing results and constraint violations at convergence. In the table the distance column conveys the Euclidean norm of the difference between observed points and fitted points. Compared to the previous example, we see an even greater improvement in the performance in the accelerated versions of the two algorithms. In general, it is safe to conclude that distance majorization is a viable alternative to its most likely fastest competitor in non-smooth convex 
optimization.

\begin{figure}
\centering
\includegraphics[scale=0.45]{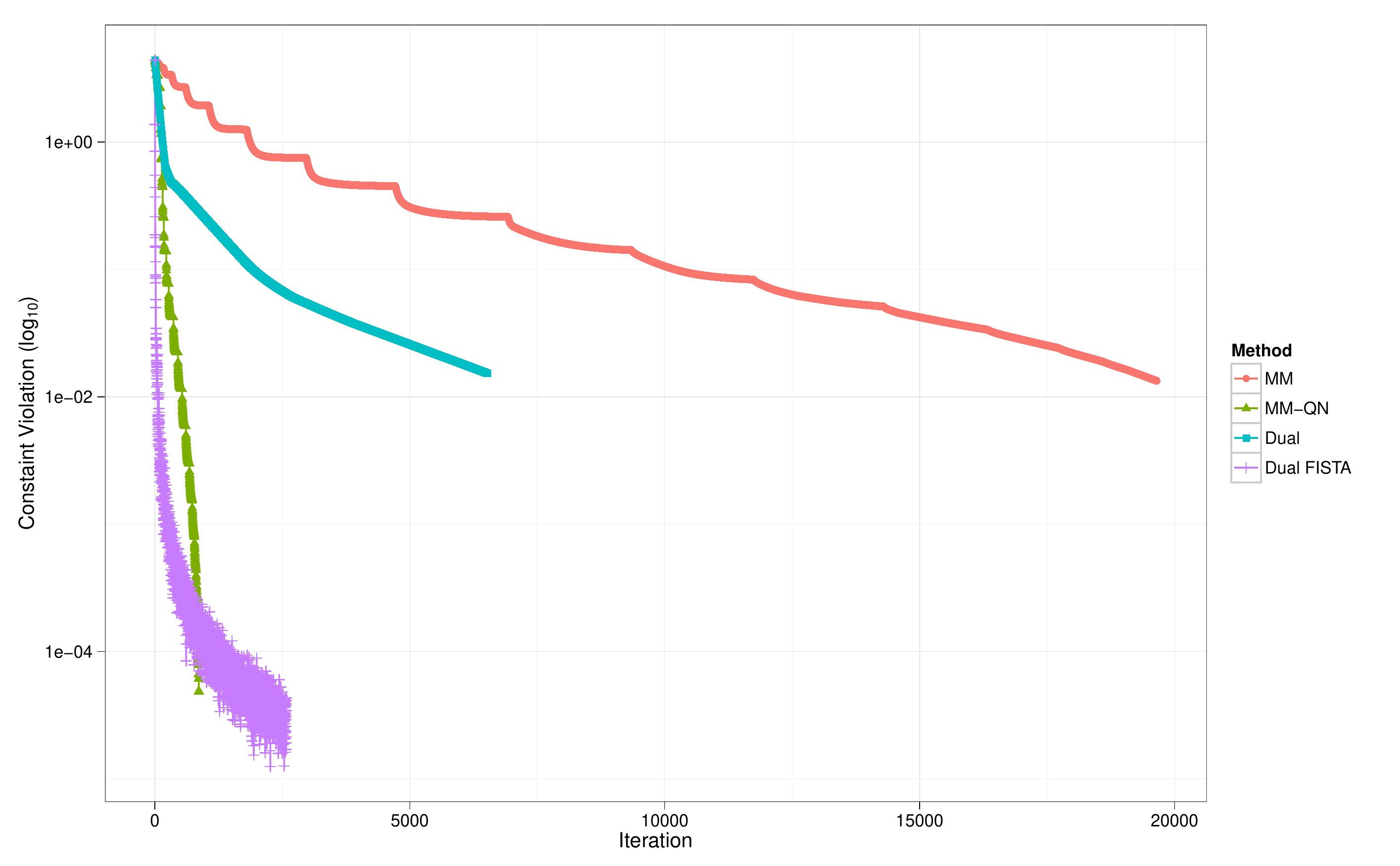}
\caption{A comparison of the MM algorithm, its quasi-Newton acceleration, the dual proximal gradient algorithm, and its FISTA acceleration applied to a univariate isotonic regression problem.
\label{fig:isotone}}
\end{figure}

% latex table generated in R 2.15.1 by xtable 1.7-0 package
% Fri Aug  3 17:16:15 2012
\begin{table}[ht]
\caption{Timing comparisons and constraint violations for the isotonic regression example.}
\label{tab:isotone}
\begin{center}
\begin{tabular}{lrrrr}
\hline\noalign{\smallskip}
  Method & Time (sec) & Iterations & Distance & Constraint Violation \\ 
\noalign{\smallskip}\hline\noalign{\smallskip}
 MM & 24.45 & 19651 & 9.633144 & -1.330351e-02 \\ 
 MM-QN & 3.27 & 863 & 9.677731 & -4.869077e-05 \\ 
 Dual & 12.70 & 6526 & 9.637778 & -1.525531e-02 \\ 
 Dual (Acc.) & 5.46 & 2578 & 9.677847 & -4.104088e-05 \\ 
   \noalign{\smallskip}\hline
\end{tabular}
\end{center}
\end{table}

%% ----------------------------------------------------------------------
%% 3. Convex Regression
%% ----------------------------------------------------------------------
\subsubsection*{Least Squares Fitting with Convex Functions}

Given responses $y_i$, predictor vectors $\bx_i$ in $\Real^p$, and case weights $w_i$, convex regression seeks to minimize the sum of squares of residuals 
\begin{eqnarray*}
    \frac 12 \sum_{i=1}^n w_i (y_i - \theta_i)^2
\end{eqnarray*}
subject to the constraints $\bxi_i^t (\bx_j - \bx_i) \le \theta_j - \theta_i$ for every ordered pair $(i,j)$
\cite{BoyVan2004}. In effect, $\theta_i$ is viewed as the value of the regression function $\theta(\bx)$ at the point $\bx_i$. The unknown vector $\bxi_i \in \Real^p$ serves as a subgradient of $\theta(\bx)$ at $\bx_i$.  Because convexity is preserved by maxima, the formula
\begin{eqnarray*}
\theta(\bx) & = & \max_j \Big[\theta_j+\bxi_j^t (\bx - \bx_j) \Big]
\end{eqnarray*}
defines a convex function with value $\theta_i$ at $\bx=\bx_i$. In concave regression the opposite constraint inequalities are imposed.  Interpolation of predicted values in this model is accomplished by simply taking minima or maxima. Estimation reduces to a positive semidefinite quadratic program involving $n(p+1)$ variables and $n(n-1)$ inequality constraints. Note that the feasible region is nontrivial because it contains the point $(\btheta, \bXi) = ({\mathbf 0}, {\mathbf 0})$, where $\bXi = [\bxi_1, \ldots, \bxi_n]$.

The penalized objective function is
\begin{eqnarray*}
   f_\mu(\btheta, \bXi) = \frac 12 \sum_{i=1}^n w_i (y_i - \theta_i)^2 + \frac{\mu}{2} \sum_{j \ne k} d_{C_{jk}}^2(\btheta, \bXi),
\end{eqnarray*}
where $C_{jk} = \{(\btheta, \bXi): \bxi_k^t (\bx_j - \bx_k) \le \theta_j - \theta_k\}$. Let $P_{C_{jk}}(\btheta, \bXi)_i$ and $P_{C_{jk}}(\btheta, \bXi)^l$ denote the components of $P_{C_{jk}}(\btheta, \bXi)$ relevant to $\theta_i$ and $\bxi_l$, respectively.  The surrogate function
\begin{eqnarray*}
	g_\mu(\btheta, \bXi \mid \btheta_m, \bXi_m)
    &=& \frac 12 \sum_{i=1}^n w_i (y_i - \theta_i)^2 + \frac{\mu}{2} \sum_{i=1}^n \sum_{j \ne k} \|\theta_i - P_{C_{jk}}(\btheta_m, \bXi_m)_i\|_2^2  \\
    & & + \frac{\mu}{2} \sum_{l=1}^n \sum_{j \ne k} \|\bxi_l - P_{C_{jk}}(\btheta_m, \bXi_m)^l\|_2^2
\end{eqnarray*}
admits the minimizer
\begin{eqnarray*}
    \theta_{m+1,i} &=& \frac{w_i}{w_i + n(n-1)\mu} y_i + \frac{\mu}{w_i+n(n-1)\mu} \sum_{j \ne k} P_{C_{jk}}(\btheta_m, \bXi_m)_i  \\
    \bxi_{m+1,l} &=& [n(n-1)]^{-1} \sum_{j \ne k} P_{C_{jk}}(\btheta_m, \bXi_m)^l.
\end{eqnarray*}

The projection operator $P_{C_{jk}}$ is easy to compute because $C_{jk}$ is a half-space.
Furthermore, if we define the quantities
\begin{eqnarray*}
r_{jj} & = & 0 \quad \mbox{and} \quad 
    r_{jk} \:\;\, = \:\;\, \left[ \frac{(\bx_j-\bx_k)^t \bxi_{k} - \theta_{j}+\theta_{k}}{2 + \|\bx_j - \bx_k\|_2^2} \right]_+ \:\; \mbox{for} \quad j \ne k,
\end{eqnarray*}
then the sums entering the MM updates reduce to 
\begin{eqnarray*}
    \sum_{j \ne k} P_{C_{jk}}(\btheta, \bXi)_i &=& n(n-1) \theta_{i} + \sum_{k=1}^n r_{ik} - \sum_{j=1}^n r_{ji}    \\
    \sum_{j \ne k} P_{C_{jk}}(\btheta, \bXi)^l &=& n(n-1) \bxi_{l} - \sum_{j=1}^n r_{jl} (\bx_j - \bx_l)
\end{eqnarray*}
evaluated at $\btheta = \btheta_m$ and $\bXi = \bXi_m$.

Figure~\ref{fig:convex-reg} displays a randomly generated data set with 51 data points and the corresponding least squares fit with convexity constraints. We employed the same geometrically increasing sequence of $\mu$ used earlier, took five secant conditions for the quasi-Newton acceleration, and set the stopping criterion (\ref{eq:stopping_rule}) to $\rho=10^{-8}$. The MM algorithm requires 8940 iterations and 4.12 seconds in total to achieve the objective value of 1.0709 and the maximal constraint violation at order of $7 \times 10^{-9}$.

\begin{figure}
\centering
\includegraphics[width=4.5in]{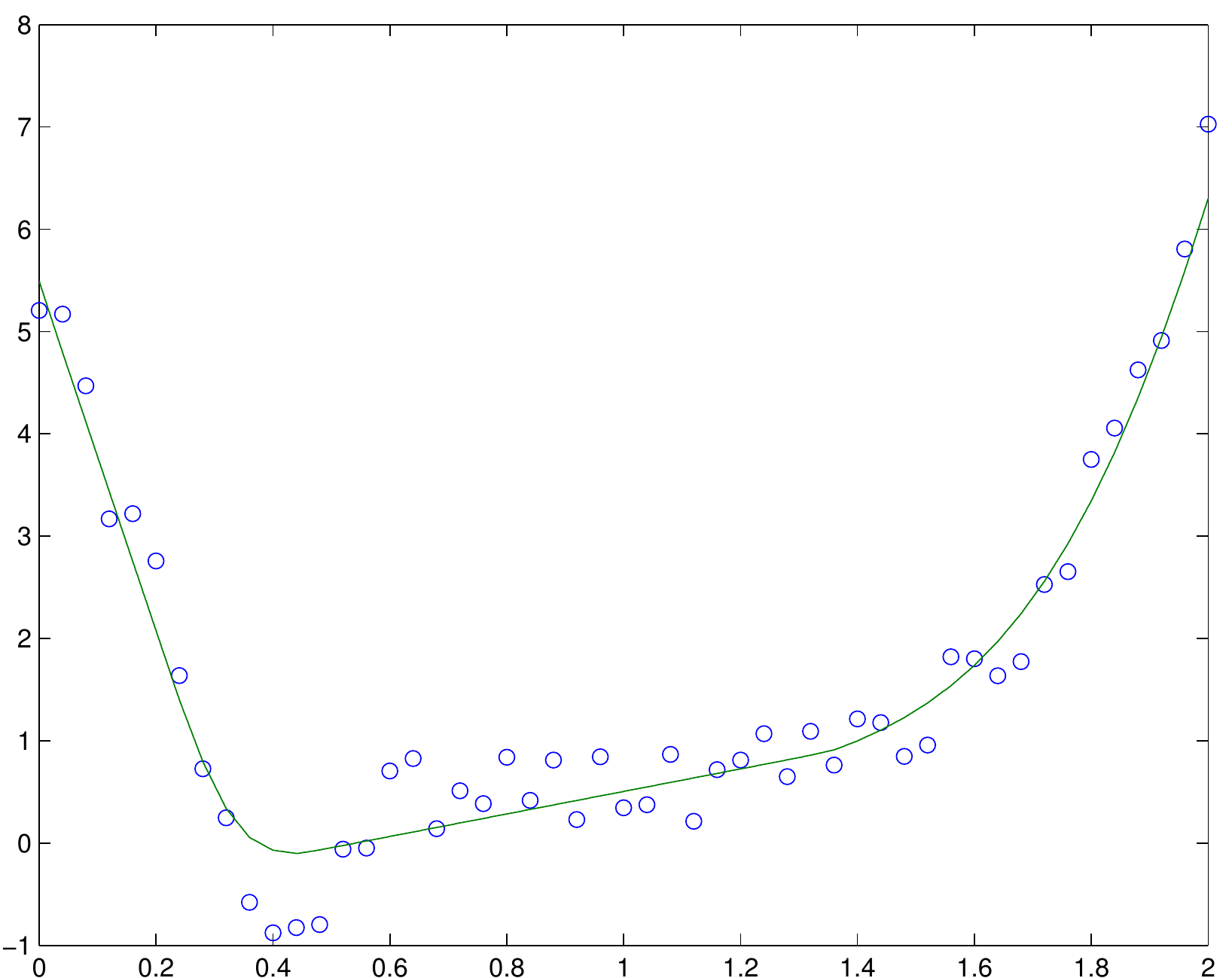}
\caption{Fitted data for convex regression.}
\label{fig:convex-reg}
\end{figure}

%% ----------------------------------------------------------------------
%% 4. SVM
%% ----------------------------------------------------------------------
\subsubsection*{Support Vector Machine}

Given data $(y_i,\bx_i)$, $i=1,\ldots,n$, where $y_i \in \{-1,1\}$ and $\bx_i \in \Real^p$, the goal of discriminant analysis is to choose classification labels $y_i$ using the $p$-dimensional predictor $\bx_i$. The support vector machine (SVM) \cite{Vap2000} is one of the most popular classifiers and
potentially benefits from distance penalization.  Here the problem is to minimize
the quadratic loss function
\begin{eqnarray*}
f(\btheta,b,\bepsilon) & = & \sum_{i=1}^n \epsilon_i + \frac{\lambda}{2} \|\btheta\|^2
\label{classification_objective_function}
\end{eqnarray*}
subject to the inequality constraints
\begin{eqnarray*}
1 - y_i (b + \bx_{i}^t \btheta) & \le & \epsilon_i \label{hyperplane_separation}
\end{eqnarray*}
using slack variables $\epsilon_i \ge 0$. See Example 15.5.2 of the book \cite{Lan2012}
for further details about problem formulation and passing to the dual. In the following we assume that the first element of $\bx_i$ is 1, and thus the intercept $b$ is absorbed in the parameter $\btheta$. Then the penalized objective function is
\begin{eqnarray*}
    f_\mu(\bepsilon,\btheta) = \sum_{i=1}^n  \epsilon_i + \frac{\lambda}{2} \|\btheta\|_2^2 + \frac{\mu}{2} \sum_{j=1}^n d_{C_j}^2(\bepsilon,\btheta)
\end{eqnarray*}
where $C_j = \{(\bepsilon,\btheta): \epsilon_j+y_j \bx_j^t \btheta \ge 1\}$. Minimizing the surrogate function
\begin{eqnarray*}
    g_\mu(\bepsilon,\btheta|\bepsilon_m, \btheta_m) 
    &=& \sum_{i=1}^n  \epsilon_i+\frac{\lambda}{2} \|\btheta\|_2^2 + \frac{\mu}{2} \sum_{j=1}^n \|\bepsilon - P_{C_j}(\bepsilon_m,\btheta_m)_{\bepsilon} \|_2^2 \\
& & + \frac{\mu}{2} \sum_{j=1}^n \|\btheta - P_{C_j}(\bepsilon_m,\btheta_m)_{\btheta}\|_2^2
\end{eqnarray*}
subject to the non-negativity of $\epsilon_i$ yields the next iterate
\begin{eqnarray*}
    \bepsilon_{m+1} &=& \frac 1n \left[ \sum_{j=1}^n P_{C_j}(\bepsilon_m,\btheta_m)_{\bepsilon} - \mu^{-1} {\bf 1}_n \right]_+  \\
    \btheta_{m+1} &=& \frac{\mu}{\lambda+n\mu} \sum_{j=1}^n P_{C_j}(\bepsilon_m,\btheta_m)_{\btheta}.
\end{eqnarray*}
Because $C_{j}$ is a half-space,
\begin{eqnarray*}
    P_{C_j}(\bepsilon_m,\btheta_m) = \begin{pmatrix}
    \bepsilon_m \\ \btheta_m
    \end{pmatrix} + \left[\frac{1 - \epsilon_{mj} - y_j\bx_j^t \btheta_m }{1 + y_j^2\|\bx_j\|_2^2} \right]_+
\begin{pmatrix}
    \be_j \\ y_j \bx_j
    \end{pmatrix},
\end{eqnarray*}
where the vector $\be_j$ has all entries equal to 0 except for a 1 at entry $j$.

We report the results on an example SVM problem with a training data set of $n=1371$ observations and $p=7$ features. We employed the same geometrically increasing sequence of $\mu$ and the same stopping criterion $\rho$ used in the previous example. At $\lambda = 10$, the MM algorithm takes 14,432 iterations and 2.69 seconds to achieve the objective value 489.0058 and the maximal constraint violation $8.6 \times 10^{-9}$.

As a generalization, consider the kernel SVM \cite{SchSmo2001} attractive in handling $p >> n$ problems. The optimization problem is to minimize
\begin{eqnarray*}
    \sum_{i=1}^n \epsilon_i + \frac{\lambda}{2} \sum_{i=1}^n \sum_{j=1}^n \theta_i \theta_j K(\bx_i,\bx_j)
\end{eqnarray*}
subject to
\begin{eqnarray*}
    1 - y_i \left[ b + \sum_{j=1}^n \theta_j K(\bx_i,\bx_j) \right] \le \epsilon_i \quad \mbox{and} \quad \epsilon_i \ge 0 \quad \mbox{ for all } i.
\end{eqnarray*}
Common choices of kernels include the polynomial kernel
\begin{eqnarray*}
    K(\bx_i,\bx_j) = \langle \bx_i ,\bx_j \rangle^\gamma
\end{eqnarray*}
and the Gaussian kernel
\begin{eqnarray*}
    K(\bx_i,\bx_j) = \exp \left\{ - \frac{\|\bx_i - \bx_j\|_2^2}{2\sigma^2} \right\}.
\end{eqnarray*}
Since $K$ is positive semi-definite, it can be expressed in terms of a Cholesky decomposition $K = LL^t$. With reparameterization $\balpha = L^t \btheta$, the problem transforms to
\begin{eqnarray*}
    \min_{b,\bepsilon,\balpha} \quad \sum_{i=1}^n \epsilon_i + \frac{\lambda}{2} \|\balpha\|_2^2
\end{eqnarray*}
subject to
\begin{eqnarray*}
    1 - y_i \left( b + \sum_{j=1}^n l_{ij} \alpha_j \right) \le \epsilon_i \quad \mbox{and} \quad \epsilon_i \ge 0 \quad \mbox{ for all } i,
\end{eqnarray*}
which is essentially the same as the original SVM. The Cholesky decomposition costs $n^3/6$ flops and might be a concern for data with huge number of observations. Some kernels used in genomics are naturally low rank with trivial Cholesky factors $L$ and $L^t$. Even for a full-rank kernel $K$, one can resort to the fast Lanczos algorithm \cite{GolVan1996} to extract its top $r$ eigen-pairs $K \approx U_r D_r U_r^t$ and set $L = U_r D_r^{1/2}$, an $n \times r$ matrix.

%% ----------------------------------------------------------------------
%% 5. Fire Station Problem
%% ----------------------------------------------------------------------
\subsection*{The Fire Station Problem}

Finally, we give another example that distance majorization need not be fettered to Euclidean distances.
Indeed, Euclidean distances may be inappropriate in some problems. Consider the problem of determining the optimal location of a new fire station in a city where the streets occur on a rectangular grid.  The station should be situated to guarantee the shortest routes to several major buildings spread throughout the city.  This is just the generalized Heron problem with the $\ell_1$ norm substituting for the Euclidean norm \cite{ChiLan2013}. More general treatment of the problem under arbitrary norms and infinite dimensions can be found in \cite{MorNam2011,MorNamSal2012,MorNamSal2011}. Here we are concerned with efficient computation with a particular norm. The projection operators $P^1_C(\bx)$ are now harder to calculate. Indeed, they are often sets
rather than points. Fortunately, when $C$ is a rectangle $[\ba,\bb]$ with sides parallel to the standard axes, $P^1_C(\bx)$ is a point with components
\begin{eqnarray*}
P^1_C(\bx)_i & = &
\begin{cases}
a_i & x_i < a_i \\ x_i & a_i \le x_i \le b_i \\ b_i & x_i > b_i.
\end{cases}
\end{eqnarray*}
To minimize the objective function, we minimize the surrogate function 
\begin{eqnarray*}
g(\bx \mid \bx_{n}) & = & \sum_{i=1}^m \sum_{j=1}^p | x_j - P^1_{C_i}(\bx_n)_j | .
\end{eqnarray*}
Because the $\ell_1$ norm separates variables, we obtain a very simple update formula.
\begin{eqnarray*}
x_{n+1,j} & = & \operatorname{median}\left[P^1_{C_1}(\bx_{n})_j, \ldots, P^1_{C_m}(\bx_{n})_j \right].
\end{eqnarray*}

Consider the example where the buildings have centers $(-7,0.5)$, $(-5,-8)$, $(4,7)$, $(5,2)$, and $(-4,6)$
and half-side lengths of 0.5.  Minimizing the sum of $\ell_1$ and $\ell_2$ distances yields the results
shown in Figure~\ref{fig:fireHouse}. The optimal position clearly depends on the underlying
norm.  For more general $\ell_1$ problems, the solution may not be unique because the projection
operator does not reduce to a single point.

\begin{figure}
\centering
\begin{tabular}{cc}
\subfloat[$\ell_1$ norm]{\includegraphics[scale=0.45]{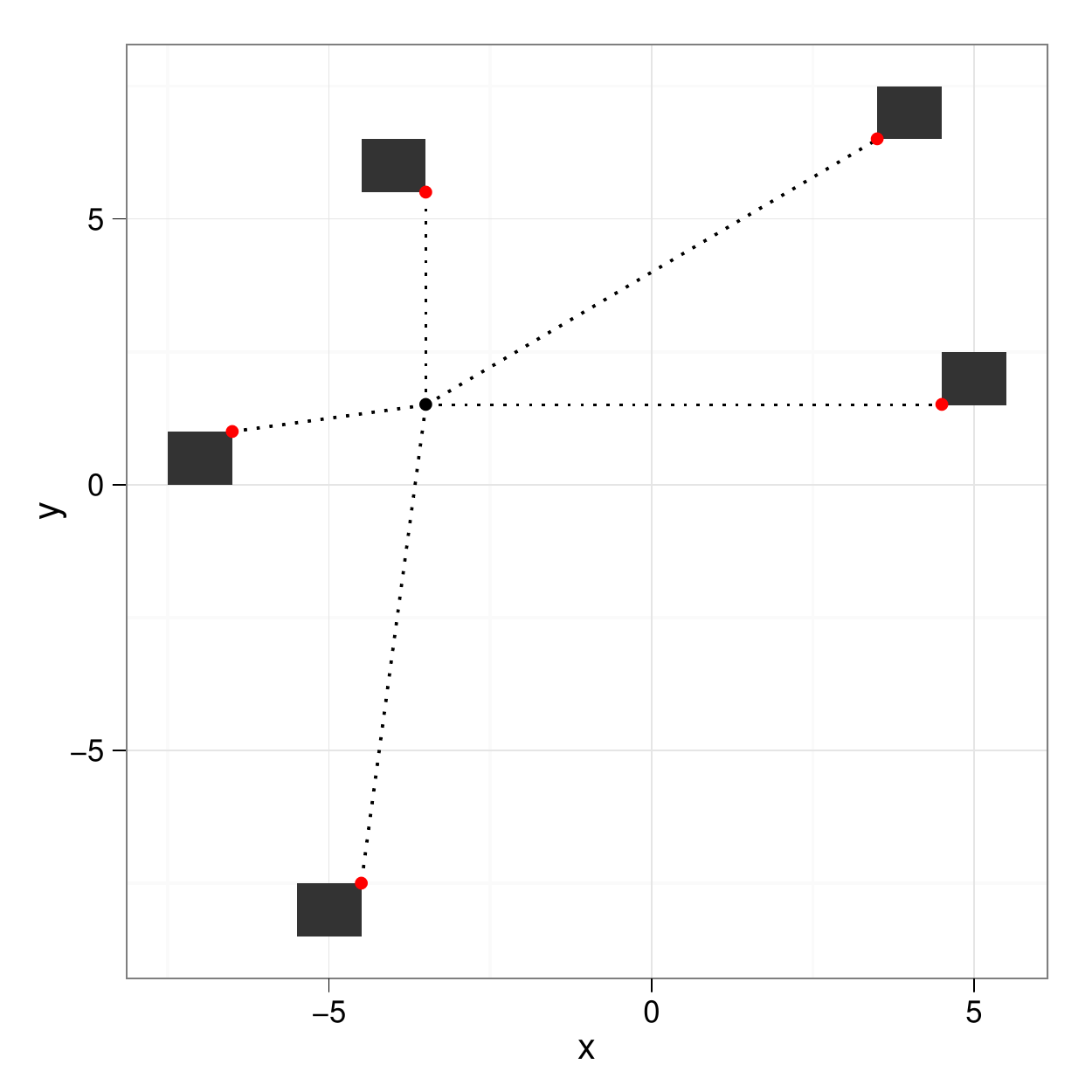}} 
& 
\subfloat[$\ell_2$ norm]{\includegraphics[scale=0.45]{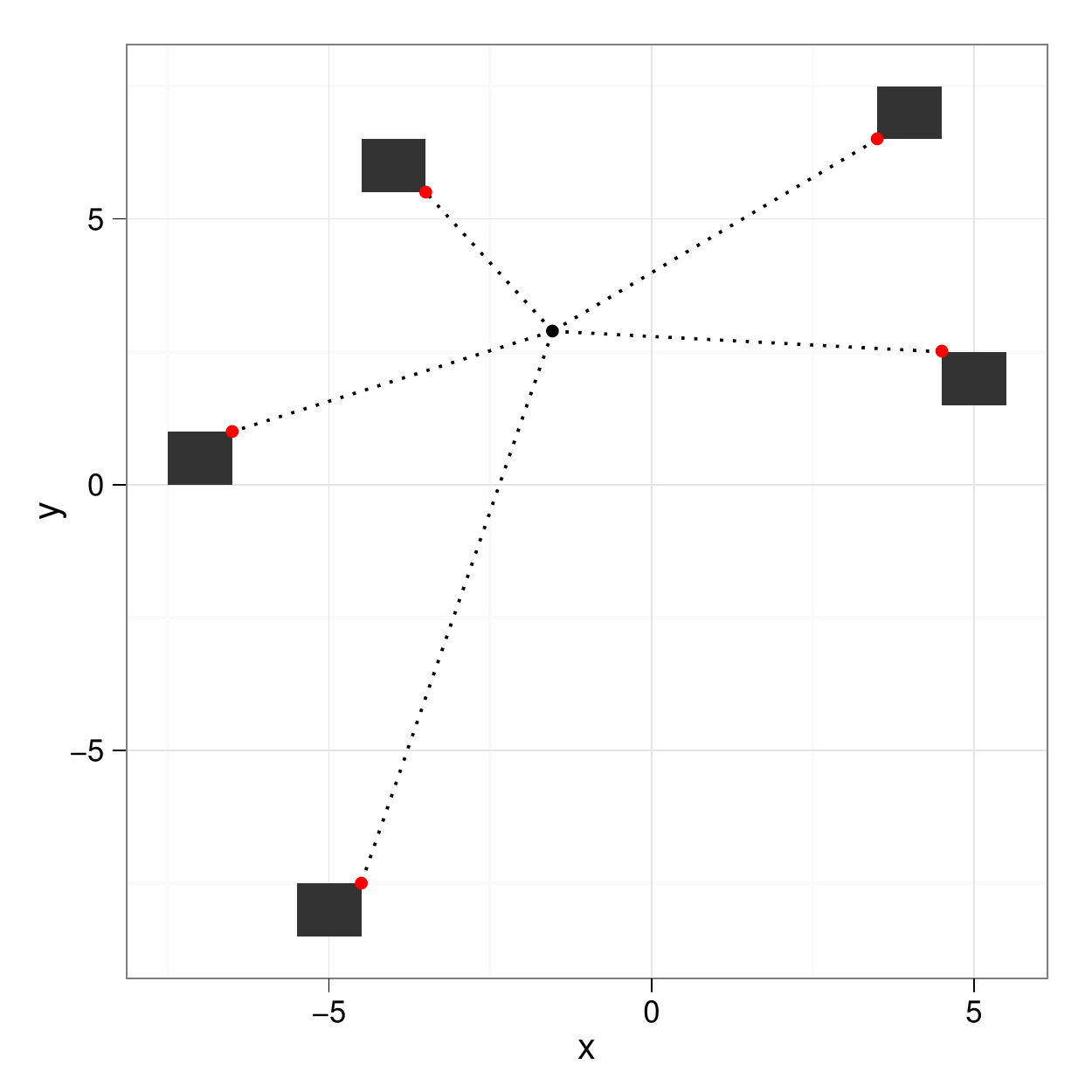}}\\
\end{tabular}
\caption{Optimal location for the fire station.}
\label{fig:fireHouse}
\end{figure}

%% ----------------------------------------------------------------------
%% Convergence Analysis
%% ----------------------------------------------------------------------
\section{Convergence Analysis}

We now prove convergence of the distance majorization algorithm under conditions pertinent to Euclidean distances. For broader impact,  we relax the convexity requirement on $\ell(\bx)$. For example, in statistics, the objective function corresponding to many widely used robust estimators are often not-convex \cite{HubRon2009}. Some of our convergence results hold for such objective functions. When $\ell(\bx)$ is convex, it is possible to prove stronger results, and we comment on what changes when convexity is assumed. Let us first consider the convergence of the MM algorithm
for solving subproblem (\ref{eq:quadratic_penalty}).
The convergence theory of MM algorithms hinges on the properties of the algorithm map $\psi(\bx) \equiv \arg\min_{\by} g(\by \mid \bx)$. 
For easy reference, we state a simple version of Meyer's monotone convergence theorem \cite{Mey1976} instrumental in proving convergence in our setting.

\begin{proposition}\label{prop:MM_limit_points}
  Let $f(\bx)$ be a continuous function on a domain $S$ and
   $\psi(\bx)$ be a continuous algorithm map from $S$ into $S$ satisfying
 $f(\psi(\bx)) < f(\bx)$ for all $\bx \in S$ with $\psi(\bx) \neq \bx$.
  Suppose for some initial point $\bx_{0}$ that the set
\begin{eqnarray*}
\mathcal{L}_f(\bx_{0}) & \equiv & \{\bx \in S : f(\bx) \leq f(\bx_{0}) \}
\end{eqnarray*}
is compact. Then
  \begin{inparaenum}[(a)]
   \item \label{part:successive_iterates}
     $\lim_{m \to \infty} \lVert \bx_{m+1} - \bx_{m} \rVert = 0$, 
  \item \label{part:fixed_points} 
    all cluster points are fixed points of $\psi(\bx)$, and
  \item \label{part:overall_convergence}
    $\bx_m$ converges to one of the fixed points if they are finite in number.
   \end{inparaenum}
\end{proposition}
The function $f(\bx)$ is the objective to be minimized. In our context, the objective function is $f_\mu(\bx) = \ell(\bx) + \frac{\mu}{2}\sum_{i=1}^m \dist(\bx,C_i)^2$. 
We make the following assumptions: \begin{inparaenum}[(a)]
\item \label{part:C1} $\ell(\bx)$ is continuously differentiable and $\ell(\bx) + \frac{\kappa}{2} \| \bx \|^2$ is convex for some constant $\kappa > 0$,
\item \label{part:f_coercive} $f_{\kappa}(\bx)$ is coercive in the sense that $\lim_{\|\bx\| \to \infty} f_{\kappa}(\bx) = \infty$, and
\item \label{part:mu_big} $\mu > \kappa$.
\end{inparaenum}
Note that $f_{\mu}(\bz)$ inherits coerciveness from $f_{\kappa}(\bz)$.
Assumption (\ref{part:f_coercive}) is met in several different scenarios, for example, if at least one of the $C_i$ is bounded or if $\ell(\bx)$ itself is coercive. When $\ell(\bx)$ is convex, $\ell(\bx) + \frac{\kappa}{2} \| \bx \|^2$ is convex for any $\kappa > 0$. Consequently, assumption (\ref{part:mu_big}) holds for any positive $\mu$. If $f(\bx)$ is non-convex, but the smallest eigenvalue of the Hessian $d^2f(\bx)$ is bounded below by 
$\lambda$, then one can take $\kappa = -\lambda$. As a rule, it can be challenging to identify $\kappa$ in advance, and consequently  in practice we would not know how large to choose  $\mu$ to ensure the conditions for convergence when $\kappa$ is unknown. Nonetheless, $\kappa$ can be explicitly determined in many useful cases. In the Appendix, we derive $\kappa$ for the classic Tukey biweight of robust estimation. 

\begin{proposition}\label{prop:subprob_convergence}
The cluster points of the MM iterates for solving subproblem (\ref{eq:quadratic_penalty}) are stationary points  of $f_\mu(\bx)$ under assumptions (\ref{part:C1}) through (\ref{part:mu_big}) above. If the number of stationary points is finite, then the MM iterates converge. Finally, if $f_\mu(\bx)$ has a unique stationary point, then the MM iterates converge to that stationary point, which globally minimizes $f_\mu(\bx)$.
\end{proposition}

\begin{proof}
We first argue that the surrogate function $g_{\mu}(\by \mid \bx)$ is strongly convex, a crucial fact invoked later. For all $\bx, \by,$ and $\bz$, Assumption (\ref{part:C1}) implies
\begin{eqnarray*}
\ell(\by) + \frac{\kappa}{2} \| \by \|^2 &\geq \ell(\bz) + \frac{\kappa}{2}\| \bz \|^2 + [\nabla \ell(\bz) + \kappa \bz]^t(\by - \bz),\\
\end{eqnarray*}
which in turn entails
\begin{equation}
\begin{split}
\label{eq:ell_convex}
\ell(\by) &\geq \ell(\bz) + \nabla \ell(\bz)^t (\by - \bz) + \kappa \left [\frac{1}{2}\| \bz \|^2 + \bz^t(\by - \bz) - \frac{1}{2} \| \by \|^2 \right ]\\
 &= \ell(\bz) + \nabla \ell(\bz)^t (\by - \bz) - \frac{\kappa}{2} \| \bz - \by \|^2. \\
\end{split}
\end{equation}
The quadratic expansion
\begin{equation}
\begin{split}
\label{eq:quad_expansion}
\| \by - P_{C_i}(\bx) \|^2 &= \| \by - \bz + \bz - P_{C_i}(\bx) \|^2 \\
&= \| \bz - P_{C_i}(\bx) \|^2 + 2[\bz - P_{C_i}(\bx)]^t(\by - \bz) + \| \by - \bz \|^2 \\
\end{split}
\end{equation}
also holds. Combining inequality (\ref{eq:ell_convex}) with equality (\ref{eq:quad_expansion}) leads to
\begin{eqnarray*}
g_\mu(\by \mid \bx) \geq g_\mu(\bz \mid \bx) + \nabla g_\mu(\bz \mid \bx)^t (\by -\bz) + \frac{\mu-\kappa}{2} \| \bz - \by \|_2^2. \\
\end{eqnarray*}
which is equivalent to the strong convexity of $\by \mapsto g_\mu(\by \mid \bx)$. In view of this result, $\by \mapsto g_\mu(\by \mid \bx)$ has a single stationary point, which is also its unique global minimizer.

We now proceed to check the conditions given in Proposition~\ref{prop:MM_limit_points}. It is easy to verify that $f_\mu(\bx)$ is continuous. We must also show that the algorithm map $\psi(\bx)$ is continuous. Take an arbitrary convergent sequence $\bx_n$ that tends to the limit $\bx$. It suffices to prove that the sequence $\by_n = \psi(\bx_n)$ tends to $\by = \psi(\bx)$. Now there exists a constant $b$ such that 
\begin{eqnarray*}
f_\mu(\bx_n) & \leq & f_\mu(\bx) + b
\end{eqnarray*}
for all $\bx_n$. In view of the descent property, we have $f_\mu(\by_n) \leq f_\mu(\bx) + b$ as well. Hence, coerciveness implies $\by_n$ is bounded. Consider any convergent subsequence $\by_{n_k}$ with limit $\bz$.
The points $\by_{n_k}$ and $\bx_{n_k}$ are related through the stationarity condition
\begin{eqnarray*}
{\mathbf 0} = \nabla \ell(\by_{n_k}) + \mu\sum_{i=1}^m [\by_{n_k} - P_{C_i}(\bx_{n_k})].
\end{eqnarray*}
Since $\ell(\bx)$ is continuously differentiable and Euclidean projections are continuous functions, taking limits gives,
\begin{eqnarray*}
{\mathbf 0} = \nabla \ell(\bz) + \mu \sum_{i=1}^m [\bz - P_{C_i}(\bx)] = \nabla g_\mu(\bz \mid \bx).
\end{eqnarray*}
Because the surrogate function $\bz \mapsto g_\mu(\bz \mid \bx)$ possesses a unique stationary point $\by$, the subsequence $\by_{n_k}$ converges to $\by$. Given this conclusion for all subsequences of the bounded sequence $\by_n$, the sequence $\by_n$ in fact converges to $\by$.

The strict descent property of $\psi(\bx)$ follows from the uniqueness of the global minimizer of $g_\mu(\by \mid \bx)$. Because $f_\mu(\bx)$ is coercive and continuous (in fact, continuously differentiable), the set $\mathcal{L}_{f_\mu}(\bx_0)$ is compact for any initial point $\bx_0$. Therefore, Proposition~\ref{prop:MM_limit_points} implies that all cluster points of the sequence $\bx_{n+1} = \psi(\bx_n)$ are fixed points.  Since $\nabla f_\mu(\bx) = \nabla g_\mu(\bx \mid \bx)$, fixed points coincide with stationary points of $f_\mu(\bx)$. If $f_\mu(\bx)$ has finitely many stationary points, conclusion (\ref{part:overall_convergence}) of Proposition~\ref{prop:MM_limit_points} implies that the iterates converge to one of the stationary points. If the coercive function $f_\mu(\bx)$ possesses a single stationary point, then that point represents a global minimum, and the MM iterates $\bx_n$ converge to it. 
\end{proof}

Observe that Proposition~\ref{prop:subprob_convergence} does not explicitly require the loss function $\ell(\bx)$ to be convex. This is in sharp contrast to the strong convexity condition on $\ell(\bx)$ needed to ensure the global convergence of the dual ascent algorithm. The convergence of the dual ascent method is discussed further in the Appendix. For a sequence of penalization parameters $\mu_k \uparrow \infty$, we intuitively expect the solutions to the penalized problems to approach a solution to  the original problem. Indeed, this is the case. We restate Theorem 17.1 in \cite{NocWri2006} in our notation.
\begin{proposition}
\label{prop:Quadratic_Penalty_Method}
Suppose each $\bx(\mu_k)$ exactly solves subproblem (\ref{eq:quadratic_penalty}), and that $\mu_k \uparrow \infty$. Then every cluster point of the sequence $\bx(\mu_k)$ is a global solution to the original problem (\ref{eq:opt_prob}).
\end{proposition}

When $\ell(\bx)$ is coercive and possesses a unique minimizer subject to the constraints, one can justify the stronger claim that the sequence $\bx(\mu_k)$ converges to the minimizer.  Under these assumptions the sequence $\bx(\mu_k)$ is bounded and possesses exactly one cluster point. Therefore, the sequence $\bx(\mu_k)$ converges to that cluster point.  Boundedness of $\bx(\mu_k)$ follows from the inequalities  
\begin{eqnarray*}
\ell[\bx(\mu)] & \leq & f_{\mu}[\bx(\mu)] \leq f_{\mu}(\by) = \ell(\by) 
\end{eqnarray*}
for any feasible point $\by$.

%% ----------------------------------------------------------------------
%% Discussion
%% ----------------------------------------------------------------------
\section{Discussion}

The MM principle is a versatile tool. Here we demonstrate how majorizing a distance function can be leveraged to solve a variety of optimization problems with non-trivial convex constraints. The resulting MM algorithms have simple update formulas that open the door to straightforward parallelization and graceful handling of large data sets. In the case of projection onto an intersection of closed convex sets, we have demonstrated that accelerated variants of the MM algorithm are competitive with the current state-of-the-art algorithms for solving non-smooth convex programs.

Several of our examples rely on the classical penalty method.  This raises the questions of how to select the ultimate penalty constant and how fast to increase it from a low starting value.  The quality of our solutions and the rate of convergence of the MM algorithms depend on these choices. We have given some rough guidelines that work well in practice, but more theoretical and empirical insight would be helpful.  We have not encountered disastrous numerical instabilities in using the penalty method, partially because all of our computations were carried out in double precision.

Distance majorization works best for Euclidean distance. This follows from the fact that explicit formulas are available for several important projection operators. For others, such as projection onto the unit simplex, fast algorithms have been devised.  Nonetheless, as the feasible point and fire station examples show, distance majorization can be applied to non-Euclidean distances.  Is it possible to devise fast MM algorithms for computing non-Euclidean distances? This is an problem area deserving more thorough study.

In the examples we considered here all constraint violations were equally weighted. In problems where some constraints are softer than others, employing nonuniform weights on the penalty terms in the objective function (\ref{eq:general_surrogate}) may be advantageous. Introducing nonuniform weights does not change our qualitative conclusions about convergence but may improve the numerical performance of the algorithm if constraint violations are weighted differently.

Another intriguing issue is the application of distance majorization to minimization of non-convex loss functions $\ell(\bx)$ over the intersection of convex sets. In statistics, many useful robust parameter estimates employ non-convex $\ell(\bx)$, for example Tukey's biweight function and more generally M-estimators \cite{HubRon2009}. Although the strongest convergence guarantees require the uniqueness of a global solution, much of the convergence theory remains intact if convexity is no longer assumed. Our convergence theory shows that the convexity assumption on $\ell(\bx)$ can be relaxed. Extending these results and constructing new practical examples are worthy targets of future research.

\begin{acknowledgements}
We thank Janet Sinsheimer for helpful feedback in the course of this work.
We also thank the anonymous referees and associate editor for their constructive
suggestions. In particular, we appreciate the detailed comments bringing sequential 
unconstrained minimization and SUMMA to our attention and highlighting its connection 
to the MM algorithm. This research was partially supported by United States Public
Health Service grants GM53275 and HG006139.
\end{acknowledgements}

%% ----------------------------------------------------------------------
%% Appendix
%% ----------------------------------------------------------------------
\section*{Appendix}

\section*{Dual ascent algorithm}

We derive a modest generalization of an iterative algorithm for the dual of the projection problem \cite{ComPes2011}. Because constructing the dual program and the associated projected gradient algorithm are exercises in modern convex analysis, we first review a few key facts from this discipline. Readers can consult the references \cite{Ber2009,BorLew2000,HirLem2004,Roc1996,Rus2006} for proofs and further background material.

The Fenchel conjugate $f^\star(\by)$ of a function $f(\bx)$ is defined as 
\begin{eqnarray*}
f^\star(\by) & = & \underset{\bx}{\sup} \left[\by^t\bx - f(\bx) \right].
\end{eqnarray*}
When $f(\bx)$ is convex and lower semicontinuous, it satisfies the biconjugate relation $f^{\star\star}(\bx) = f(\bx)$.
In particular, the conjugate of the indicator function $\delta_C(\bx)$ of a set $C$ is the support function 
\begin{eqnarray*}
\delta_C^\star(\by) & = &\underset{\bx \in C}{\sup} \:\: \by^t\bx
\end{eqnarray*}
of $C$.  When $C$ is closed and convex, $\delta_C^{\star\star}(\bx) = \delta_C(\bx)$.

Recall that a function $f(\bx)$ is strongly convex with parameter $\eta > 0$ if the difference
$f(\bx) - \frac{\eta}{2}\| \bx \|^2_2$ is convex. Thus, a strongly
convex function has a curvature bounded away from zero. If $f(\bx)$ is strongly convex, then the
value $f^\star(\by)$ is attained at a single point $\bx$.  In this case, $f^\star(\by)$ is differentiable with gradient $\nabla f^\star(\by) = \bx$. Furthermore, $\nabla f^\star(\by)$ satisfies the Lipschitz inequality
\begin{eqnarray*}
\| \nabla f^\star(\bz) - \nabla f^\star(\by) \|_2 & \leq & \frac{1}{\eta} \| \bz - \by \|_2.
\end{eqnarray*}
Lipschitz continuity ensures global convergence of the proximal gradient algorithm for solving the dual problem. The proximity-operator $\operatorname{prox}_h(\bz)$ associated with a function $h(\bx)$ is defined as 
\begin{eqnarray*}
\prox_h(\bz) & = & \underset{\bx}{\arg\min}\; \left[ h(\bx) + \frac{1}{2} \lVert \bz - \bx \rVert_2^2 \right].
\end{eqnarray*}
Here the right hand side has a unique minimizer whenever $h(\bx)$ is convex and lower semicontinuous. 
The proximal gradient method \cite{Nes2007} is guaranteed to minimize the function $f(\bx) + g(\bx)$ when
$f(\bx)$ is differentiable, convex, and has a Lipschitz continuous gradient, and $g(\bx)$ is lower-semicontinuous and convex. The proximal gradient method iterates according to 
\begin{eqnarray*}
\bx^{n+1} & = & \prox_{\sigma g}\left[\bx^{n} - \sigma\nabla f(\bx^{n}) \right], 
\end{eqnarray*}
where $\sigma$ denotes a step size and $\bx^{n}$ the $n$th iterate. We recover the classic gradient descent method when $g(\bx)$ is the zero function, and we recover the projected gradient algorithm when $g(\bx) = \delta_C(\bx)$ is the indicator of a closed convex set $C$. Thus, the proximal gradient algorithm generalizes two important algorithm classes.

We are now ready to derive an iterative algorithm for solving the dual program of interest. Consider the slightly more general problem of minimizing a strongly convex function $f(\bx)$ over the intersection of a finite collection of closed convex sets $C_1, \ldots, C_m$. This problem can be reposed as minimizing the
function
\begin{eqnarray*}
f(\bx) + \sum_{i=1}^m \delta_{C_i}(\bx_i)
\end{eqnarray*}
subject to the constraints $\bx_1 = \bx, \ldots, \bx_m = \bx$. The Lagrangian for this problem 
\begin{eqnarray*}
L(\bx, \bx_1, \ldots, \bx_m, \bz_1, \ldots, \bz_m) 
& =  & f(\bx) - \Big(\sum_{i=1}^m \bz_i\Big)^t\bx + \sum_{i=1}^m \left[\delta_{C_i}(\bx_i) + \bz_i^t\bx_i\right].
\end{eqnarray*}
gives rise to the dual function
\begin{eqnarray*}
\mathcal{D}(\bz) 
& = & - f^\star(\bz_1 + \cdots + \bz_m) - \sum_{i=1}^m h_i(\bz_i), \\
\end{eqnarray*}
where $\bz =  (\bz_1, \ldots, \bz_m)$ denotes the concatenation of the dual variables $\bz_i$ and $h_i(\bz_i)$ is the support function of the set $C_i$ at $-\bz_i$. Thus, the dual problem of maximizing $\mathcal{D}(\bz)$ is equivalent to minimizing $f^\star(\bz_1 + \cdots + \bz_m) + \sum_{i=1}^m h_i(\bz_i)$. Given that $f(\bx)$ is strongly convex, $f^\star(\bz)$ is differentiable and in fact $\nabla f^\star(\bz)$ is Lipschitz continuous. Therefore, the dual is a prime candidate to be solved via the proximal gradient method.  Since $\sum_i h_i(\bz_i)$ separates the variable $\bz_i$, the dual proximal gradient step can be computed blockwise as
\begin{eqnarray}
\label{eq:proxgrad}
\bz^{n+1}_i & = & \prox_{\sigma h_i} \left[\bz^{n}_i - \sigma\nabla f^\star(\bz^{n}_1 + \cdots + \bz^{n}_m) \right],
\end{eqnarray}
where $\sigma$ denotes a step size. The algorithm simplifies further by applying the Moreau decomposition \cite[Lemma 2.10]{ComWaj2005}.
\begin{equation}
\label{eq:Moreau_decomp}
\bu = \prox_{\sigma h_i}(\bu) + \sigma\prox_{h_i^\star/\sigma}(\bu/\sigma).
\end{equation}
Note that $\prox_{h_i^\star/\sigma}(\bu/\sigma) = - P_{C_i}(-\bu/\sigma)$ and $\nabla f^\star(\bs)$ is a minimizer of the convex function $f(\bx) - \bs^t \bx$. Combining these identities with (\ref{eq:proxgrad}) and (\ref{eq:Moreau_decomp}) gives the algorithm
\begin{eqnarray*}
\bx^{n} & = &  \underset{\bx}{\arg\min} \;\left[ f(\bx) - (\bz^{n}_1 + \cdots + \bz^{n}_m)^t \bx \right]\\
\bz^{n+1}_i & = &  \bz^{n}_i + \sigma \left[ P_{C_i}(\bx^{n} - \sigma^{-1}\bz^{n}_i) - \bx^{n} \right]. \\
\end{eqnarray*}
Convergence is assured by setting the step length $\sigma = \eta$, where $1/\eta$ is the Lipschitz constant of 
$\nabla f^\star(\bz)$. Thus, the strong convexity condition on $f(\bx)$ is actually required for convergence, since a closed, convex function $f$ is Lipschitz continuous if and only if its conjugate function is strongly convex \cite{KakShaTew2009}.

The Nesterov acceleration mentioned earlier requires just a minor adjustment.
The first two iterates are computed as above; subsequent updates use the following extrapolation steps.
\begin{eqnarray*}
\bx^{n} & = &  \underset{\bx}{\arg\min} \;\left[f(\bx) - (\bz^{n}_1 + \cdots + \bz^{n}_m)^t \bx\right] \\
\bs^{n} & = & \bz^{n}_i + \frac{n-2}{n+1} [\bz^{n} - \bz^{n-1}] \\
\bz^{n+1}_i &= & \bs^{n}_i + \sigma [ P_{C_i}(\bx^{n} - t^{-1}\bs^{n}_i) - \bx^{n}].
\end{eqnarray*}

\section*{An MM formulation that fails the SUMMA condition}

Consider minimizing the univariate function $f(x) = \cos(x)$. According to the quadratic upper bound principle \cite{BLin1988}, the function
\begin{eqnarray*}
g(y \mid x) & = & \cos(x) - \sin(x)(y-x) + \frac{1}{2} (y - x)^2
\end{eqnarray*}
majorizes $f(y)$. The MM algorithm $x_{n} = \psi(x_{n-1})$ employs the iteration map
\begin{eqnarray}
\label{eq:mapping}
\psi(x) & = & x + \sin(x).
\end{eqnarray}
Figure~\ref{fig:cosine_summa} depicts the first two majorizations starting $x_0 = 1$. The global SUMMA condition requires that 
\begin{eqnarray}
\label{eq:summa_mm_ex}
g(x \mid x_{0}) - g(x_1 \mid x_{0}) & \geq & g(x \mid x_1) - f(x)
\end{eqnarray}
for all $x$, but Figure~\ref{fig:cosine_summa} shows that this inequality fails for some $x$. Restricted to the interval $C = [\frac{\pi}{2}, \frac{3\pi}{2}]$, however, the MM algorithm does belong to the SUMMA class.  It is geometrically obvious that all iterates reside in $C$ when iteration commences there.
Furthermore, the objective function is convex, and the MM algorithm reduces to gradient descent with a fixed step size that is no greater than twice the inverse of the Lipschitz constant of $f'(x)$.  In these circumstances Byrne \cite{Byr2008a} verifies the SUMMA condition. 

On the other hand, one can prove convergence without invoking the SUMMA machinery. The MM algorithm has
fixed points at integer multiples of $\pi$.  Even multiples correspond to maxima and odd multiples to minima. The maxima are repelling, and the minima are attracting. If the algorithm starts on the interval 
$[2k \pi, 2(k+1)\pi]$, then it remains there. Hence, the hypotheses of Proposition~\ref{prop:MM_limit_points}
are met. Although the SUMMA condition is helpful in forcing convergence and understanding the rate of convergence, there is no need to compel majorization to satisfy it.

\begin{figure}
\centering
\begin{tabular}{cc}
\subfloat[An MM algorithm for minimizing $f(x) = \cos(x)$. The first two majorizations are shown when $x_0=1$.]{\label{fig:cosine_mm}
\includegraphics[scale=0.45]{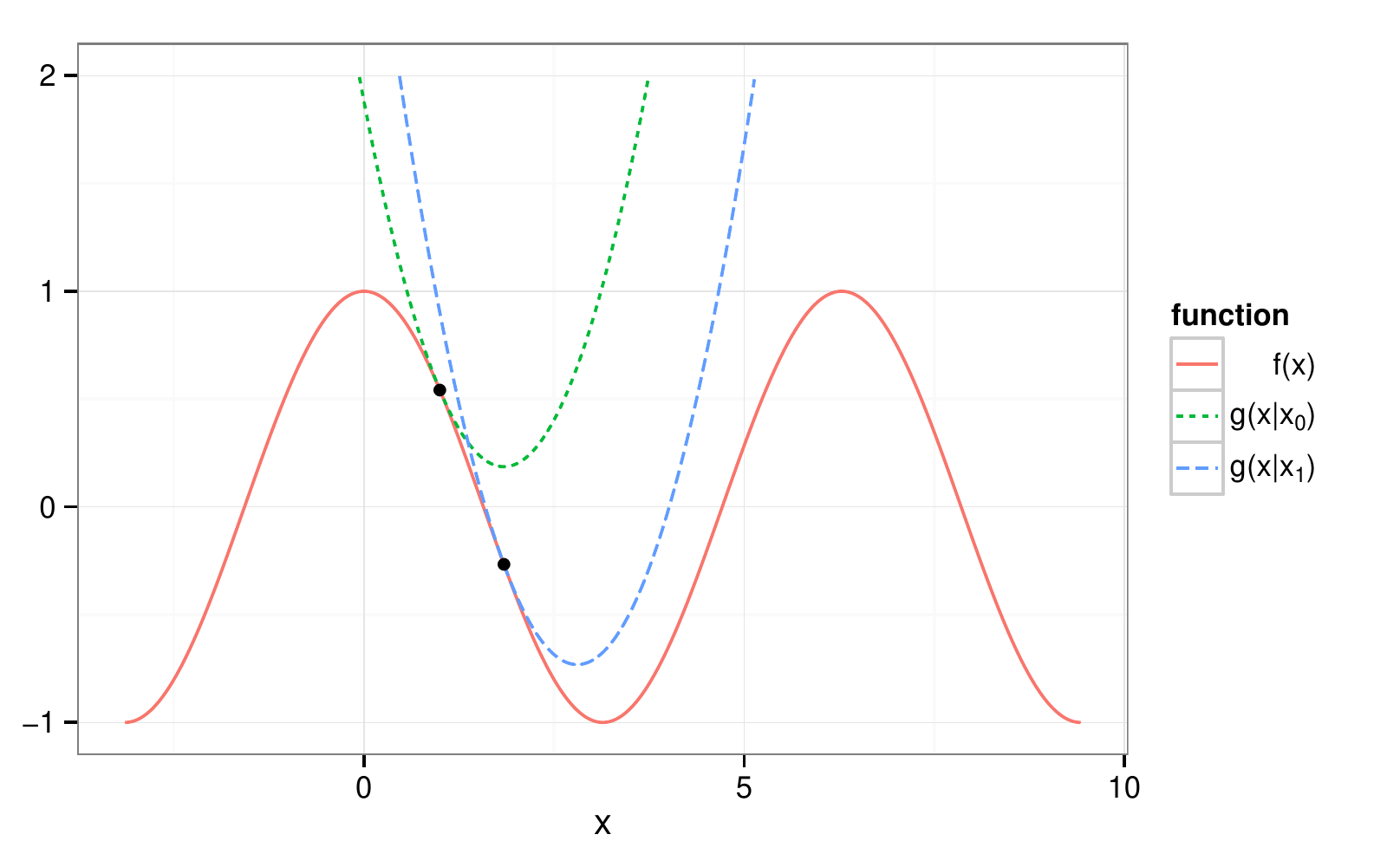}} 
\\
\subfloat[Violation of the SUMMA inequality.  The solid line plots the left hand side (LHS) of the SUMMA inequality (\ref{eq:summa_mm_ex}). The dashed line plots the right hand side (RHS) of the SUMMA inequality (\ref{eq:summa_mm_ex}). The SUMMA inequality requires the dashed line to never cross above the solid line.]{\label{fig:cosine_summa}
\includegraphics[scale=0.45]{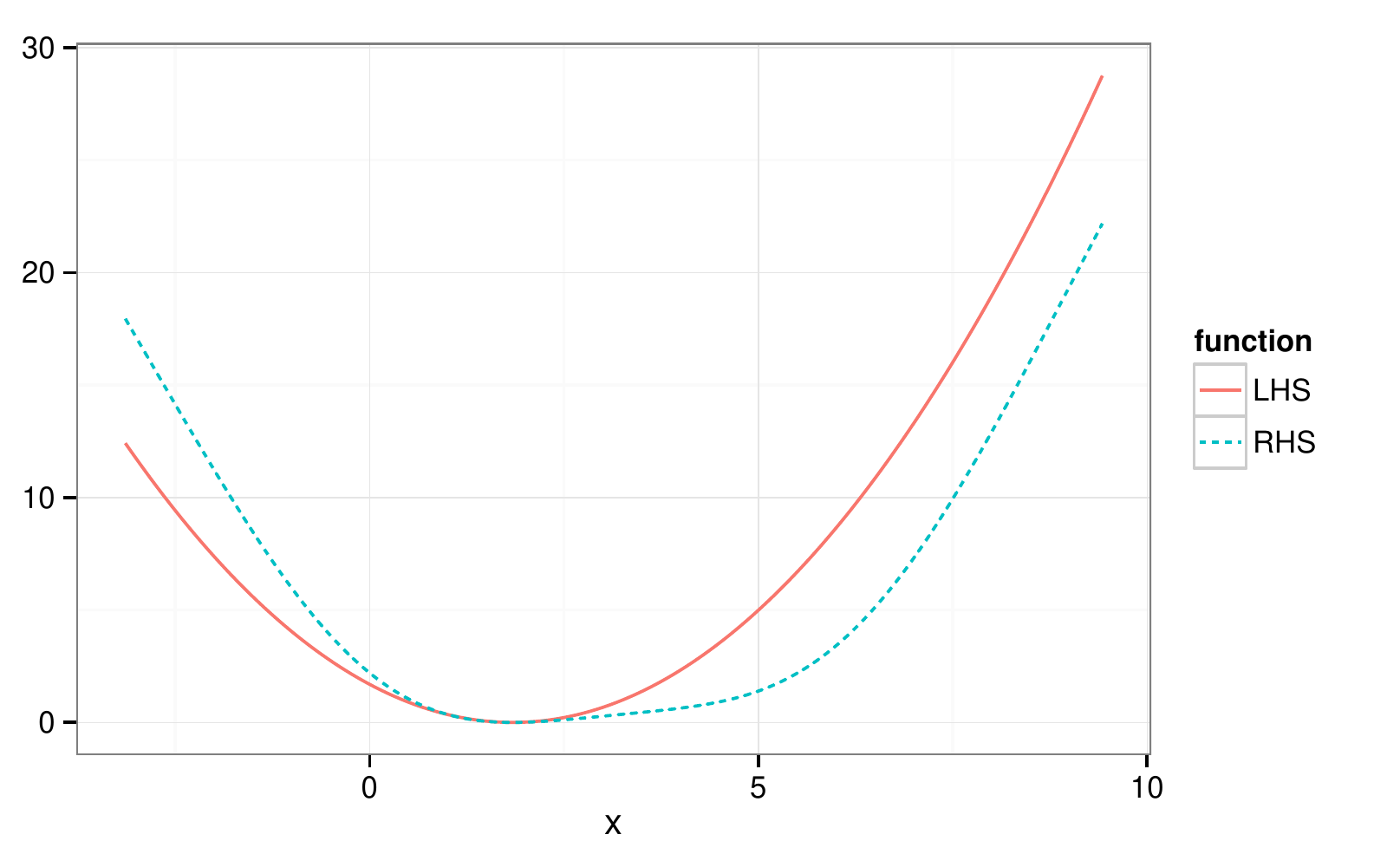}}\\
\end{tabular}
\caption{An example of an MM algorithm whose auxiliary functions fail the SUMMA condition.}
\label{fig:cosine_example}
\end{figure}

\section*{Tukey's Biweight}

In robust linear regression, outliers are a major concern. In standard regression one minimizes 
the squared error loss
\begin{eqnarray*}
f(\bbeta) & = & \sum_{i=1}^n \phi(y_i - \bx_i^t \beta),
\end{eqnarray*}
with $\phi(t) = \frac{1}{2}t^2$, where $y_i$ is the response of case $i$, $\bx_i$ is the predictor vector for 
case $i$, and $\bbeta$ is the vector of regression coefficients. One can moderate the influence of outliers
by substituting Tukey's biweight \cite{HubRon2009}
\begin{eqnarray}
\label{eq:tukey}
\phi(t) = \begin{cases}
\frac{c^2}{6}\left\{1 - \left[1 - \left(\frac{t}{c}\right)^2\right]^3\right\} & \text{if $|t| \leq c$} \\
c^2/6 & \text{if $|t| > c$.}
\end{cases}
\end{eqnarray}
for $\phi(t) = \frac{1}{2}t^2$. The new loss determined by the function (\ref{eq:tukey}) discounts the contribution of residuals $y_i - \bx_i^t \beta$ whose absolute value exceeds $c$. To calculate a global
lower bound $\lambda$ on the eigenvalues of $d^2f(\bx)$, we note that 
\begin{eqnarray*}
d^2f(\bbeta) & = & \sum_{i=1}^n \phi''(y_i - \bx_i^t\bbeta)\bx_i\bx_i^t,
\end{eqnarray*}
where
\begin{eqnarray*}
\phi''(t) = \begin{cases}
1 - 6(\frac{t}{c})^2 + 5(\frac{t}{c})^4 & \text{if $|t| \leq c$} \\
0 & \text{if $|t| > c$.}
\end{cases}
\end{eqnarray*}
The function $\phi''(t)$ achieves a minimum of $-\frac{4}{5}$ at $t = \pm c \sqrt{\frac{3}{5}}$. It follows
that we can take
\begin{eqnarray*}
\kappa & \geq & \frac{4}{5} \rho(\bX^t\bX),
\end{eqnarray*}
where $\bX$ is the matrix with columns $\bx_i$ and $\rho(\bM)$ denotes the largest eigenvalue of the symmetric matrix $\bM$. Interestingly, $\kappa$ does not depend on $c$. Similar calculations can be carried out for other robust choices of $\phi(t)$.

% BibTeX users please use one of
%\bibliographystyle{spbasic}      % basic style, author-year citations
\bibliographystyle{spmpsci}      % mathematics and physical sciences
\bibliography{DistanceMajorization}   % name your BibTeX data base

\end{document}